\documentclass[a4paper]{amsart}

\usepackage{amsmath,amsthm,amssymb}
\usepackage[all]{xy} \SelectTips{cm}{}
\usepackage{graphicx}
\usepackage{enumitem}

\usepackage{tensor}
\usepackage{ifthen}
\usepackage{mathtools}
\usepackage{bm}

\usepackage[usenames,dvipsnames]{xcolor}

\usepackage{tikz}
\usetikzlibrary{cd}

\usepackage{hyperref}
\hypersetup{
    colorlinks,
    linkcolor={red!50!black},
    citecolor={blue!50!black},
    urlcolor={blue!80!black}
}
\usepackage{aliascnt}[2009/09/08] % to get hyperref's \autoref to work properly with counters
\newcommand{\harxiv}[1]{\href{http://arxiv.org/abs/#1}{\texttt{arXiv:#1}}}
\newcommand{\hyref}[2]{\hyperref[#2]{#1~\ref*{#2}}}
\newcommand{\bib}[6]{{\bibitem{#2} #3: {\emph{#4},} #5#6.}}

\newcommand{\Exc}[1]{\mathcal{E}xc(#1)}    % set of full exceptional sequences
\newcommand{\mExc}[1]{m\mathcal{E}xc(#1)}  % subset of module sequences
\newcommand{\ASS}{{\rm(\dag)}}
           
\newcommand{\xra}{\xrightarrow{~\alpha~}}
\newcommand{\xlb}{\xleftarrow{~\beta~}}
\newcommand{\lfdict}{\\[0.9ex]} % line feed in dictionary table of Appendix

 \newcommand{\IN}{{\mathbb{N}}} 
\newcommand{\IP}{{\mathbb{P}}}

 \newcommand{\IZ}{{\mathbb{Z}}}

\newcommand{\cA}{{\mathcal A}}  \newcommand{\cC}{{\mathcal C}}
\newcommand{\cD}{{\mathcal D}} \newcommand{\cE}{{\mathcal E}} 
 \newcommand{\cH}{{\mathcal H}} 
  
  \newcommand{\cO}{{\mathcal O}}
  
\newcommand{\cS}{{\mathcal S}}

\newcommand{\pd}{\mathrm{pd}}                    % projective dimension
\newcommand{\kk}{{\mathbf{k}}}

\newcommand{\op}{^{\mathrm{{op}}}}

\newcommand{\orth}{^\perp}

\newcommand{\udim}{\underline{\mathrm{dim}}}

\newcommand{\coloneqq}{\mathrel{\mathop:}=}

\DeclareMathOperator{\cone}{\mathrm{cone}}

\DeclareMathOperator{\Hom}{\mathrm{Hom}}
\DeclareMathOperator{\RHom}{\mathrm{RHom}}
\DeclareMathOperator{\Aut}{\mathrm{Aut}}
\DeclareMathOperator{\Ext}{\mathrm{Ext}}
\DeclareMathOperator{\ext}{\mathrm{ext}}
\DeclareMathOperator{\id}{{\mathrm{id}}}
\DeclareMathOperator{\rk}{{\mathrm{rk}}}

\DeclareMathOperator{\End}{{\mathrm{End}}}

\newcommand{\TTT}{\mathsf{T}\!}   % special case for the \TTT_X notation

\newcommand{\SSS}{\mathsf{S}}

\newcommand{\Coh}{\mathrm{Coh}}

\newcommand{\lmod}[1]{#1\mathrm{\text{-}mod} }

\newcommand{\DD}{\mathcal{D}}
\newcommand{\Db}{\DD^b}

\newcommand{\RR}{\mathsf{R}}

\newcommand{\extension}[1]{\left( \begin{array}{@{}c@{}} #1 \end{array} \right)}
\newcommand{\smextension}[2]{{\textstyle\binom{#1}{#2}}}

\newcommand{\symm}[1]{\mathsf{Sym}(#1)}
\newcommand{\braid}[1]{\mathsf{Br}(#1)}

\newcommand{\onto}{\to\!\!\!\!\!\to}
\newcommand{\into}{\hookrightarrow}

\newcommand{\too}{\longrightarrow}
\newcommand{\isom}{ \xrightarrow{ \,\smash{\raisebox{-0.65ex}{\ensuremath{\scriptstyle\sim}}}\,}}
\newcommand{\blank}{-}

\newcommand{\xxto}[1]{\xrightarrow{\raisebox{-0.2ex}{\ensuremath{{\scriptscriptstyle #1}}}}} 
\newcommand{\xxfrom}[1]{\xleftarrow{\raisebox{-0.2ex}{\ensuremath{{\scriptscriptstyle #1}}}}}

\newtheorem{theorem}{Theorem}[section]

\newtheorem*{theorem*}{Theorem}
\newtheorem*{definition*}{Definition}

  \newaliascnt{proposition}{theorem}
  \newtheorem{proposition}[proposition]{Proposition}
  \aliascntresetthe{proposition}

  \newaliascnt{lemma}{theorem}
  \newtheorem{lemma}[lemma]{Lemma}
  \aliascntresetthe{lemma}

  \newaliascnt{corollary}{theorem}
  \newtheorem{corollary}[corollary]{Corollary}
  \aliascntresetthe{corollary}

\theoremstyle{definition}

  \newaliascnt{definition}{theorem}
  \newtheorem{definition}[definition]{Definition}
  \aliascntresetthe{definition}

  \newaliascnt{remark}{theorem}
  \newtheorem{remark}[remark]{Remark}
  \aliascntresetthe{remark}

  \newaliascnt{question}{theorem}
  
  \aliascntresetthe{question}

  \newaliascnt{example}{theorem}
  \newtheorem{example}[example]{Example}
  \aliascntresetthe{example}

\newtheorem*{conjecture*}{Conjecture}

\usepackage{tikz}
\usetikzlibrary{cd,backgrounds}

\newcommand{\captionnode}[2]{
  \hspace{-0.8em}%
  \smash{\tikz[baseline=-0.5ex]{\draw[#1] (0,0) node {#2}}}%
  \hspace{-0.4em} }

\newcommand{\worm}[1]{
  \begin{tikzpicture}[scale=0.4,thick,every node/.style={thick,fill,circle,inner sep=1pt}]
  #1
  \end{tikzpicture}
}

\newcommand{\wormgridtwo}[1]{
  \begin{tikzpicture}[scale=0.4]
    \draw[gray!50] (0,1) -- (0,0) -- (1,0);
    \begin{scope}[thick,every node/.style={thick,fill,circle,inner sep=1pt}]
     #1
    \end{scope}
\end{tikzpicture}
}

\newcommand{\wormgrid}[1]{
  \begin{tikzpicture}[scale=0.4]
    \draw[gray!50] (0,2) -- (0,0) -- (2,0);
    \draw[gray!50] (0,1) -- (1,1) -- (1,0);
    \begin{scope}[thick,every node/.style={thick,fill,circle,inner sep=1pt}]
     #1
    \end{scope}
\end{tikzpicture}
}

\newcommand{\wormgridc}[2]{
  \begin{tikzpicture}[scale=0.4]
    \draw[gray!50] (0,2) -- (0,0) -- (2,0);
    \draw[gray!50] (0,1) -- (1,1) -- (1,0);
    \begin{scope}[thick,every node/.style={thick,fill,circle,inner sep=1pt}]
     #1
    \end{scope}
    #2
\end{tikzpicture}
}

\newcommand{\wormcap}[1]{
     \node at (1,-1.5) {$\binom{123}{#1}$};
}

\newcommand{\ltranspose}[1]{\raisebox{8.0ex}{$\xleftrightarrow{(#1)\cdot}$}}

\newcommand{\thorm}[4]{% #1,#2 = offset within tikzpicture; #3 = permutation for node,
  \def\w{#4} % the three worms, sorted by length: longest one first, shortest one last!
  \begin{scope}[shift={(#1,1.2*#2)},scale=0.4]
    \node(#3) at (1,1) {};
    \draw[gray!50] (0,2) -- (0,0) -- (2,0);
    \draw[gray!50] (0,1) -- (1,1) -- (1,0);
    \begin{scope}[thick,every node/.style={thick,fill,circle,inner sep=1pt}]
      \draw (\w[0],\w[1]) node {} -- (\w[2],\w[3]) node {} -- (\w[4],\w[5]) node {};
      \draw (\w[6],\w[7]) node {} -- (\w[8],\w[9]) node {};
      \draw (\w[10],\w[11]) node {};
    \end{scope}
    \node at (1,-1) {$\binom{123}{#3}$};  
  \end{scope}
}

  \tikzset{T1/.style={ultra thick,red}}
  \tikzset{T2/.style={ultra thick,blue}}
  \tikzset{T3/.style={ultra thick,green}}

  \tikzset{R1/.style={ultra thick,orange}}
  \tikzset{R2/.style={ultra thick,magenta}}
  \tikzset{R3/.style={ultra thick,teal}}

  \newcommand*{\clp}{0.4}  % length for size of clipping area around worm diagrams

\newcommand{\form}[4]{% #1,#2 = offset within tikzpicture; #3 = permutation for node
  \def\w{#4} % the four forms, sorted by length: longest one first, shortest one last!
  \begin{scope}[shift={(1.2*#1,1.1*#2)},scale=0.4]
    \node(#3) at (1,1) {};
    \fill[gray!0,rounded corners]
       (-\clp,-\clp) -- (3+\clp,-\clp) -- (3+\clp,\clp) -- (\clp,3+\clp) -- (-\clp,3+\clp) -- cycle;
    \draw[gray!50] (0,3) -- (0,0) -- (3,0);
    \draw[gray!50] (0,2) -- (1,2) -- (1,0);
    \draw[gray!50] (0,1) -- (2,1) -- (2,0);
    \begin{scope}[thick,every node/.style={thick,fill,circle,inner sep=1pt}]
      \draw (\w[0],\w[1]) node {} -- (\w[2],\w[3]) node {} -- (\w[4],\w[5]) node {} -- (\w[6],\w[7]) node {};
      \draw (\w[8],\w[9]) node {} -- (\w[10],\w[11]) node {} -- (\w[12],\w[13]) node {};
      \draw (\w[14],\w[15]) node {} -- (\w[16],\w[17]) node {};
      \draw (\w[18],\w[19]) node {};
    \end{scope}
  \end{scope}
}

\begin{document}
\title[Exceptional sequences]
      {Exceptional Sequences and Spherical Modules \newline
       for the Auslander Algebra of $\bm{\kk[x]/(x^t)}$}

\author{Lutz Hille und David Ploog}

\begin{abstract}
We classify spherical modules and full exceptional sequences of modules over the Auslander algebra of $\kk[x]/(x^t)$. We categorify the left and right symmetric group actions on these exceptional sequences to two braid group actions: of spherical twists along simple modules, and of right mutations. In particular, every such exceptional sequence is obtained by spherical twists from a standard sequence, and likewise for right mutations.
\end{abstract}

\begingroup\renewcommand\thefootnote{}%
\footnote{MSC 2010: 16D90, 16G20, 16S38, 18E30}
% 14F05: Algebraic geometry: Cohomology theory: sheaves and derived categories
% 16D90: Associative algebras: Modules, bimodules, ideals: module categories
% 16E35: Associative algebras: Homological methods: Derived categories
% 16G20: Associative algebras: Representations theory of rings and algebras: Representations of quivers
% 16S38: Associative algebras: Various constructions: Rings from non-commutative algebraic geometry
% 18E30: Category theory: Abelian category: Derived categories
%
\footnote{Keywords: Auslander algebra, full exceptional sequence, exceptional module, spherical module}
\footnote{Contact: \texttt{lhill\_01@uni-muenster.de},
                   \texttt{dploog@math.fu-berlin.de}}
\addtocounter{footnote}{-2}\endgroup

\maketitle

\setcounter{tocdepth}{1}
{\centering\parbox{0.7\textwidth}{\tableofcontents}\par}
\addtocontents{toc}{\protect\setcounter{tocdepth}{-1}}     % No toc entry for Introduction

\section*{Introduction}\label{Sintro}
\addtocontents{toc}{\protect\setcounter{tocdepth}{1}}      % but enable toc entries for other sections

\noindent
In this text, we study the Auslander algebras $\cA_t$ of $\kk[x]/(x^t)$. This family of finite-dimensional algebras is well-known in representation theory. It also occurs for certain matrix problems, i.e.\ actions of linear groups on flags \cite[\S4]{Hille-Roehrle}. In previous work \cite{Hille-Ploog1}, we link $\cA_t$ to $(t-1)$-chains of $(-2)$-curves on projective surfaces.

We classify exceptional and spherical modules over $\cA_t$ in \hyref{Theorem}{thm:module-classification}, and full exceptional sequences of modules in \hyref{Theorem}{thm:exceptional-bijections}. Below, we just state the enumerative consequences of these classification results:

\begin{theorem*} The numbers of the following types of objects over the algebra $\cA_t$ are:

  \begin{tabular}{@{\hspace{1em}}ll}
    $2^t-1$   & exceptional modules; \\
    $2^t-1-t$ & $2$-spherical modules; \\
    $t!$      & full exceptional sequences of modules.
  \end{tabular}
\end{theorem*}

Moreover, we describe full exceptional sequences of modules combinatorially by \emph{worm diagrams} in \hyref{Subsection}{cor:worm-diagrams}. In \hyref{Theorem}{thm:exceptional-bijections}, we establish a natural bijection of these diagrams with the symmetric group. Hence this group acts, from both sides, on worm diagrams, i.e.\ on full exceptional sequences of $\cA_t$-modules.

We categorify both symmetric group actions to braid group actions. For the left action, this is done by the twist functors along the spherical simple modules; see \hyref{Subsection}{sub:braids-twisting}. For the right action, we use right mutations of exceptional sequences; see \hyref{Subsection}{sub:braids-mutation}. On the following two pages, we show these braid group actions for $\cA_4$.

\newpage

\centerline{\textbf{Right mutation graph of size 4 worm diagrams}}

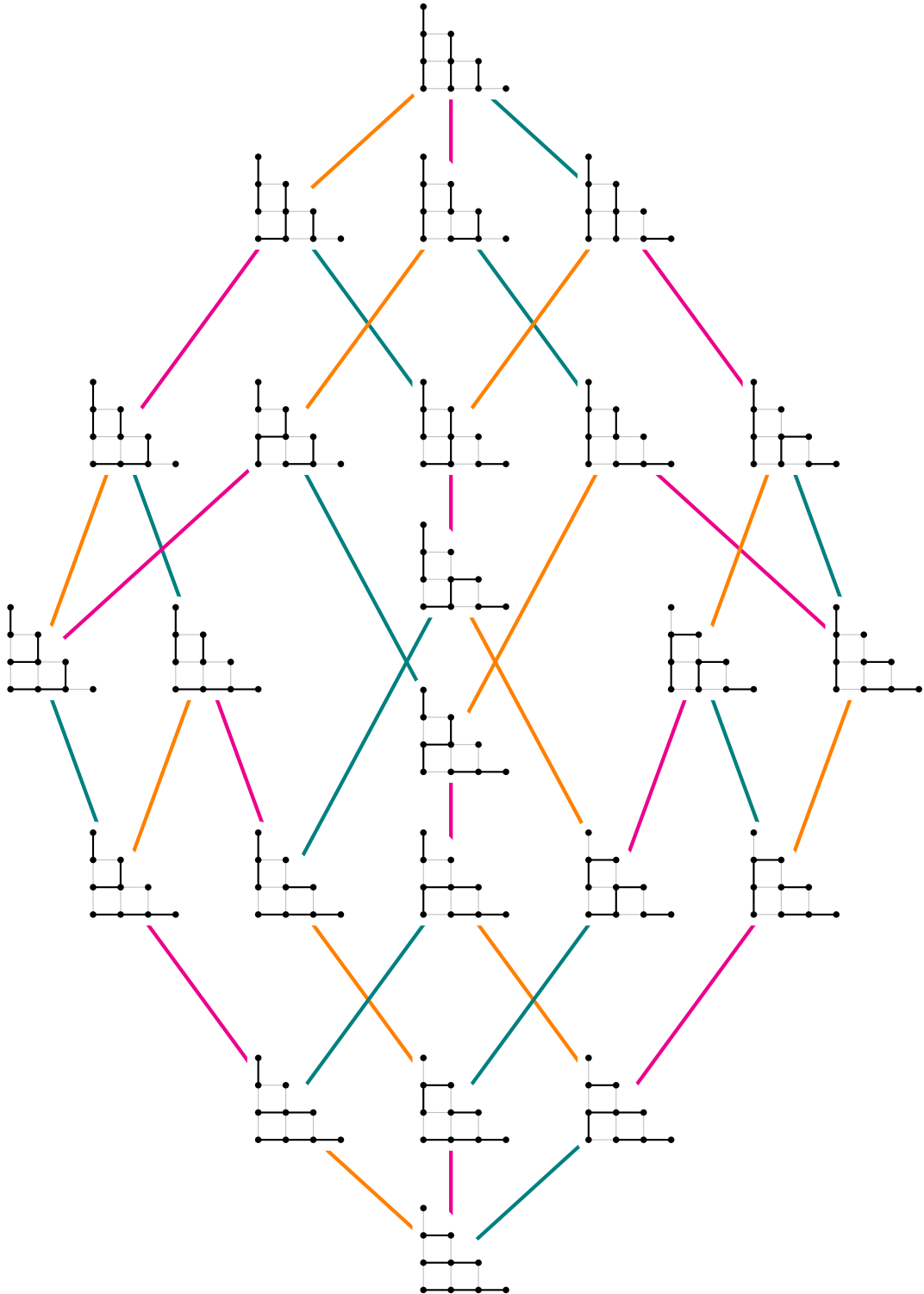
\vfill

\hspace*{-1.25cm}
\noindent
%\fbox{
\begin{tikzpicture}[scale=1.05]
% f=6:
  \form{ 0}{   8}{4321}{{0,0,0,1,0,2,0,3, 1,0,1,1,1,2, 2,0,2,1, 3,0}}
% f=5:
  \form{-2}{   6}{3421}{{0,0,1,0,1,1,1,2, 0,1,0,2,0,3, 2,0,2,1, 3,0}}
  \form{ 0}{   6}{4231}{{0,0,0,1,0,2,0,3, 1,0,2,0,2,1, 1,1,1,2, 3,0}}
  \form{ 2}{   6}{4312}{{0,0,0,1,0,2,0,3, 1,0,1,1,1,2, 2,0,3,0, 2,1}}
% f=4
  \form{-4}{   3}{3241}{{0,0,1,0,2,0,2,1, 0,1,0,2,0,3, 1,1,1,2, 3,0}}
  \form{ 0}{   3}{3412}{{0,0,1,0,1,1,1,2, 0,1,0,2,0,3, 2,0,3,0, 2,1}}
  \form{-2}{   3}{2431}{{0,0,0,1,1,1,1,2, 1,0,2,0,2,1, 0,2,0,3, 3,0}}
  \form{ 2}{   3}{4213}{{0,0,0,1,0,2,0,3, 1,0,2,0,3,0, 1,1,1,2, 2,1}}
  \form{ 4}{   3}{4132}{{0,0,0,1,0,2,0,3, 1,0,1,1,2,1, 2,0,3,0, 1,2}}
% f=3:
  \form{-3}{   0}{3214}{{0,0,1,0,2,0,3,0, 0,1,0,2,0,3, 1,1,1,2, 2,1}}
  \form{ 5}{   0}{4123}{{0,0,0,1,0,2,0,3, 1,0,2,0,3,0, 1,1,2,1, 1,2}}
  \form{ 0}{-1.1}{2413}{{0,0,0,1,1,1,1,2, 1,0,2,0,3,0, 0,2,0,3, 2,1}}
  \form{ 0}{ 1.1}{3142}{{0,0,1,0,1,1,2,1, 0,1,0,2,0,3, 2,0,3,0, 1,2}}
  \form{-5}{   0}{2341}{{0,0,1,0,2,0,2,1, 0,1,1,1,1,2, 0,2,0,3, 3,0}}
  \form{ 3}{   0}{1432}{{0,0,0,1,0,2,1,2, 1,0,1,1,2,1, 2,0,3,0, 0,3}}
% f=2:
  \form{-4}{  -3}{2314}{{0,0,1,0,2,0,3,0, 0,1,1,1,1,2, 0,2,0,3, 2,1}}
  \form{-2}{  -3}{3124}{{0,0,1,0,2,0,3,0, 0,1,0,2,0,3, 1,1,2,1, 1,2}}
  \form{ 0}{  -3}{2143}{{0,0,0,1,1,1,2,1, 1,0,2,0,3,0, 0,2,0,3, 1,2}}
  \form{ 2}{  -3}{1342}{{0,0,1,0,1,1,2,1, 0,1,0,2,1,2, 2,0,3,0, 0,3}}
  \form{ 4}{  -3}{1423}{{0,0,0,1,0,2,1,2, 1,0,2,0,3,0, 1,1,2,1, 0,3}}
% f=1:
  \form{-2}{  -6}{2134}{{0,0,1,0,2,0,3,0, 0,1,1,1,2,1, 0,2,0,3, 1,2}}
  \form{ 0}{  -6}{1324}{{0,0,1,0,2,0,3,0, 0,1,0,2,1,2, 1,1,2,1, 0,3}}
  \form{ 2}{  -6}{1243}{{0,0,0,1,1,1,2,1, 1,0,2,0,3,0, 0,2,1,2, 0,3}}
% f=0:
  \form{ 0}{  -8}{1234}{{0,0,1,0,2,0,3,0, 0,1,1,1,2,1, 0,2,1,2, 0,3}}

% arrows:
  \begin{scope}[on background layer]
% f=6:
  \draw[R1] (4321) -- (3421); \draw[R2] (4321) -- (4231); \draw[R3] (4321) -- (4312);
% f=5:
  \draw[R2] (3421) -- (3241); \draw[R3] (3421) -- (3412);
  \draw[R1] (4231) -- (2431); \draw[R3] (4231) -- (4213);
  \draw[R1] (4312) -- (3412); \draw[R2] (4312) -- (4132);
% f=4:
  \draw[R1] (4213) -- (2413); \draw[R2] (4213) -- (4123);
  \draw[R1] (3241) -- (2341); \draw[R3] (3241) -- (3214);
  \draw[R2] (3412) -- (3142);
  \draw[R1] (4132) -- (1432); \draw[R3] (4132) -- (4123);
  \draw[R2] (2431) -- (2341); \draw[R3] (2431) -- (2413);
% f=3:
  \draw[R2] (2413) -- (2143);
  \draw[R1] (3214) -- (2314); \draw[R2] (3214) -- (3124);
  \draw[R1] (4123) -- (1423);
  \draw[R2] (1432) -- (1342); \draw[R3] (1432) -- (1423);
  \draw[R3] (2341) -- (2314);
  \draw[R1] (3142) -- (1342); \draw[R3] (3142) -- (3124);
% f=2:
  \draw[R1] (3124) -- (1324);
  \draw[R2] (1423) -- (1243);
  \draw[R1] (2143) -- (1243); \draw[R3] (2143) -- (2134);
  \draw[R2] (2314) -- (2134);
  \draw[R3] (1342) -- (1324);
% f=1:
  \draw[R1] (2134) -- (1234);
  \draw[R2] (1324) -- (1234);
  \draw[R3] (1243) -- (1234);
  \end{scope}
\end{tikzpicture}
%}  % \fbox end

\vfill

\noindent
The $24$ worm diagrams of size 4, i.e.\ all full exceptional sequences of $\cA_4$-modules.
Downwards edges are right mutations:
  \captionnode{R1}{$\RR_1$}, \captionnode{R2}{$\RR_2$}, \captionnode{R3}{$\RR_3$}.
See \hyref{Subsection}{sub:braids-mutation}.

\newpage

\centerline{\textbf{Spherical twist graph of size 4 worm diagrams}}

\vfill

\hspace*{-1.25cm}
\noindent
%\fbox{
\begin{tikzpicture}[scale=1.05]
% f=6:
  \form{ 0}{   8}{4321}{{0,0,0,1,0,2,0,3, 1,0,1,1,1,2, 2,0,2,1, 3,0}}
% f=5:
  \form{ 2}{   6}{4312}{{0,0,0,1,0,2,0,3, 1,0,1,1,1,2, 2,0,3,0, 2,1}}
  \form{ 0}{   6}{4231}{{0,0,0,1,0,2,0,3, 1,0,2,0,2,1, 1,1,1,2, 3,0}}
  \form{-2}{   6}{3421}{{0,0,1,0,1,1,1,2, 0,1,0,2,0,3, 2,0,2,1, 3,0}}
% f=4
  \form{ 4}{   3}{4213}{{0,0,0,1,0,2,0,3, 1,0,2,0,3,0, 1,1,1,2, 2,1}}
  \form{ 2}{   3}{4132}{{0,0,0,1,0,2,0,3, 1,0,1,1,2,1, 2,0,3,0, 1,2}}
  \form{ 0}{   3}{3412}{{0,0,1,0,1,1,1,2, 0,1,0,2,0,3, 2,0,3,0, 2,1}}
  \form{-2}{   3}{3241}{{0,0,1,0,2,0,2,1, 0,1,0,2,0,3, 1,1,1,2, 3,0}}
  \form{-4}{   3}{2431}{{0,0,0,1,1,1,1,2, 1,0,2,0,2,1, 0,2,0,3, 3,0}}
% f=3:
  \form{ 3}{   0}{3214}{{0,0,1,0,2,0,3,0, 0,1,0,2,0,3, 1,1,1,2, 2,1}}
  \form{ 5}{   0}{4123}{{0,0,0,1,0,2,0,3, 1,0,2,0,3,0, 1,1,2,1, 1,2}}
  \form{ 0}{ 1.1}{2413}{{0,0,0,1,1,1,1,2, 1,0,2,0,3,0, 0,2,0,3, 2,1}}
  \form{ 0}{-1.1}{3142}{{0,0,1,0,1,1,2,1, 0,1,0,2,0,3, 2,0,3,0, 1,2}}
  \form{-5}{   0}{2341}{{0,0,1,0,2,0,2,1, 0,1,1,1,1,2, 0,2,0,3, 3,0}}
  \form{-3}{   0}{1432}{{0,0,0,1,0,2,1,2, 1,0,1,1,2,1, 2,0,3,0, 0,3}}
% f=2:
  \form{ 4}{  -3}{3124}{{0,0,1,0,2,0,3,0, 0,1,0,2,0,3, 1,1,2,1, 1,2}}
  \form{-2}{  -3}{1423}{{0,0,0,1,0,2,1,2, 1,0,2,0,3,0, 1,1,2,1, 0,3}}
  \form{ 0}{  -3}{2143}{{0,0,0,1,1,1,2,1, 1,0,2,0,3,0, 0,2,0,3, 1,2}}
  \form{ 2}{  -3}{2314}{{0,0,1,0,2,0,3,0, 0,1,1,1,1,2, 0,2,0,3, 2,1}}
  \form{-4}{  -3}{1342}{{0,0,1,0,1,1,2,1, 0,1,0,2,1,2, 2,0,3,0, 0,3}}
% f=1:
  \form{ 2}{  -6}{2134}{{0,0,1,0,2,0,3,0, 0,1,1,1,2,1, 0,2,0,3, 1,2}}
  \form{ 0}{  -6}{1324}{{0,0,1,0,2,0,3,0, 0,1,0,2,1,2, 1,1,2,1, 0,3}}
  \form{-2}{  -6}{1243}{{0,0,0,1,1,1,2,1, 1,0,2,0,3,0, 0,2,1,2, 0,3}}
% f=0:
  \form{ 0}{  -8}{1234}{{0,0,1,0,2,0,3,0, 0,1,1,1,2,1, 0,2,1,2, 0,3}}

% arrows:
  \begin{scope}[on background layer]
% f=6:
  \draw[T1] (4321) -- (3421); \draw[T2] (4321) -- (4231); \draw[T3] (4321) -- (4312);
% f=5:
  \draw[T2] (4312) -- (4213); \draw[T1] (4312) -- (3412);
  \draw[T1] (4231) -- (3241); \draw[T3] (4231) -- (4132);
  \draw[T2] (3421) -- (2431); \draw[T3] (3421) -- (3412);
% f=4:
  \draw[T1] (4213) -- (3214); \draw[T3] (4213) -- (4123);
  \draw[T2] (3241) -- (2341); \draw[T3] (3241) -- (3142);
  \draw[T2] (3412) -- (2413);
  \draw[T1] (4132) -- (3142); \draw[T2] (4132) -- (4123);
  \draw[T1] (2431) -- (2341); \draw[T3] (2431) -- (1432);
% f=3:
  \draw[T1] (2413) -- (2314); \draw[T3] (2413) -- (1423);
  \draw[T2] (3214) -- (2314); \draw[T3] (3214) -- (3124);
  \draw[T1] (4123) -- (3124);
  \draw[T1] (1432) -- (1342); \draw[T2] (1432) -- (1423);
  \draw[T3] (2341) -- (1342);
  \draw[T2] (3142) -- (2143);
% f=2:
  \draw[T2] (3124) -- (2134);
  \draw[T1] (1423) -- (1324);
  \draw[T1] (2143) -- (2134); \draw[T3] (2143) -- (1243);
  \draw[T3] (2314) -- (1324);
  \draw[T2] (1342) -- (1243);
% f=1:
  \draw[T3] (2134) -- (1234);
  \draw[T2] (1324) -- (1234);
  \draw[T1] (1243) -- (1234);
  \end{scope}
\end{tikzpicture}
%}  % \fbox end

\vfill

\noindent
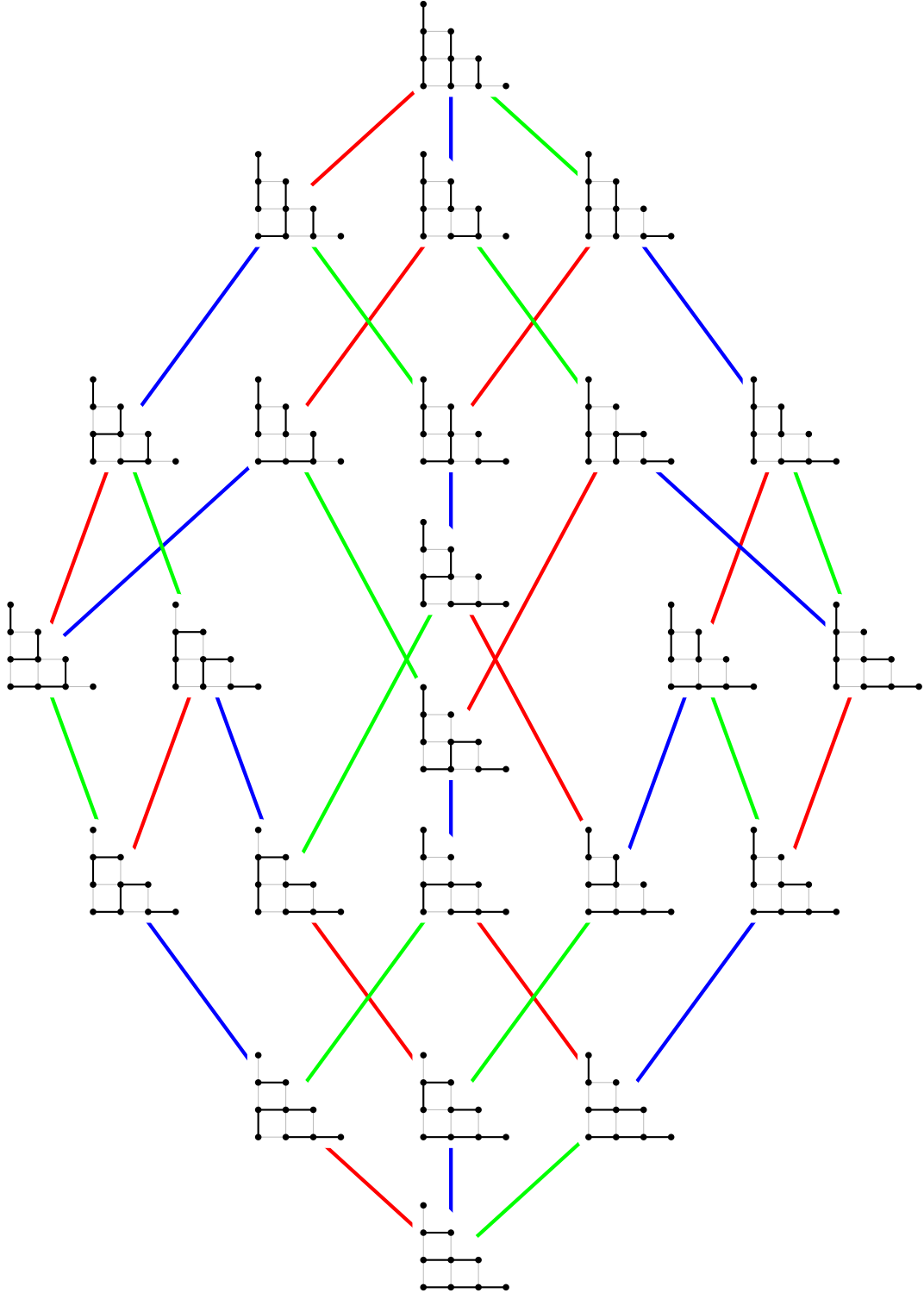
The $24$ worm diagrams of size 4, i.e.\ all full exceptional sequences of $\cA_4$-modules.
Downwards edges are spherical twists:
  \captionnode{T1}{$\TTT_{S(1)}$}, \captionnode{T2}{$\TTT_{S(2)}$}, \captionnode{T3}{$\TTT_{S(3)}$}.
See \hyref{Subsection}{sub:braids-twisting}.
%
%The schematic is a categorification of the diagrams of \cite[\S4]{Iyama-Zhang}.
%Its seven lines correspond to exceptional sequences with $f=6,\ldots,0$; see \hyref{Proposition}{prop:flattening-worms}.

\newpage

We mention related classification results. On terminology: a \emph{tilting module} $T$ in this text means a generator of the module category with $\Ext^{>0}(T,T)=0$. When the additional condition $\pd(T)\leq1$ is included, we say \emph{classical tilting module}.

The algebra $\cA_t$ is quasi-hereditary, and by \cite{BHRR}, there are $t!$ tilting modules which are $\Delta$-filtered. This class of modules coincides with classical tilting modules, and with $\tau$-tilting modules; the latter have been studied in \cite{Iyama-Zhang}. See also \cite[Ex.~5.8]{Eisele-Janssens-Raedschelder}.

Next, there are $2^{t+1}-t-2$ \emph{bricks} over $\cA_t$, i.e.\ modules $M$ with $\Hom(M,M)=\kk$. This has been worked out as follows: the algebra $\cA_t$ is $\tau$-tilting finite, and for any such algebra, there is a bijection between indecomposable $\tau$-rigid modules and bricks; see \cite[Thm.~4.1]{Demonet-Iyama-Jasso}. 
Now bricks for $\cA_t$ are in bijection with bricks of the preprojective algebra of type $A_t$,
  % From an email by Osamu Iyama: Bricks of $A=\cA_t$ are precisely bricks of the
  % preprojective algebra $B$ of type $A_n$ since $B=A/(x)$ holds and all bricks
  % of $A$ are annihilated by $x$ (dual argument to Eisele+Janssens+Raedschelders).
  % Next, bricks of $B$ are parametrised by so-called "join-irreducible elements"
  % in the Weyl group.
and those in turn have been parametrised by join-irreducible elements of the Weyl group \cite[Theorems~1.1,1.2]{Iyama-Reading-Reiten-Thomas}. The combinatorics for type $A$ are in \cite[\S 6.1]{Iyama-Reading-Reiten-Thomas}. We give a quick count of bricks in \hyref{Corollary}{cor:count}.

\emph{Semibricks}, i.e.\ sets of Hom-orthogonal bricks, have been classified in \cite{Asai}; there are $(t+1)!$ of them. They correspond to \emph{support $\tau$-tilting modules} \cite{Iyama-Zhang}.

\subsubsection*{Conventions}
We fix an algebraically closed field $\kk$. All algebras, categories and functors are over $\kk$.
Occasionally, we abbreviate dimensions of Hom spaces as $\hom(A,B) \coloneqq \dim \Hom(A,B)$, and likewise for $\Ext^i$.
The shift (or translation, or suspension) functor of a triangulated category is denoted $[1]$.
For objects $A,B$ of a triangulated category, we write $\Hom^\bullet(A,B) \coloneqq \bigoplus_{i\in\IZ}\Hom(A,B[i])[-i]$ fortheir Hom complex; it is a complex of vector spaces with zero differential.

Modules are always left modules. We compose arrows in a path algebra like functions (from right to left).
Given a vertex $i$ for a bound quiver, we denote the corresponding simple, projective and injective modules by $S(i), P(i), I(i)$, respectively.
If $M$ is a module, i.e.\ a representation of the bound quiver, we write $M_i$ for the vector space of $M$ at the vertex $i$.

The symmetric (also: permutation) group on $t$ letters is denoted by $\symm{t}$. It is generated by simple transpositions, which we write $\tau_i \coloneqq (i,i+1)$. The longest word is $\omega \coloneqq (t,\ldots,2,1)$. The braid group on $t$ strands is denoted $\braid{t}$.

\subsubsection*{Acknowledgments}
We thank the Mathematical Research Institute Oberwolfach where a part of this work was done during a Research in Pairs programme. We are also grateful to Osamu Iyama and Martin Kalck for valuable comments.

\section{Definition and basic properties of $\cA_t$}

\subsection{The algebra $\cA_t$}
The algebra $\cA_t$ is defined as the path algebra of the quiver with $t$ vertices and $2t-2$ arrows
\[ \begin{tikzcd}
  1     \ar[r, shift left, "\alpha"] & 
  2     \ar[r, shift left, "\alpha"] \ar[l, shift left, "\beta"] & 
  3     \ar[r, shift left, "\alpha"] \ar[l, shift left, "\beta"] & 
 \cdots \ar[r, shift left, "\alpha"] \ar[l, shift left, "\beta"] & 
  t-1   \ar[r, shift left, "\alpha"] \ar[l, shift left, "\beta"] & 
  t                                  \ar[l, shift left, "\beta"]
\end{tikzcd} \]
bound by a zero relation $\beta\alpha=0$ at $1$, and commutativity relations $\alpha\beta=\beta\alpha$ at intermediate vertices $2,\ldots,t-1$. We emphasise that there is no relation at $t$; this distinguishes $\cA_t$ from the preprojective algebra of the $A_t$-quiver.

As is well known, $\cA_t$ occurs as the Auslander algebra of the ring $R \coloneqq \kk[x]/(x^t)$. This means that $R$ has finitely many indecomposable (finitely generated) modules ---these are $M(i) \coloneqq R/(X^i) = \kk[X]/(X^i)$ for $i=1,\ldots,t$--- and that $\cA_t$ is the endomorphism algebra of their direct sum: $\cA_t = \End_R(M(1)\oplus\cdots\oplus M(t))$.

Moreover, $\cA_t$ also occurs as the endomorphism algebra of a very special tilting object of geometric nature. See \hyref{Appendix}{app:geometry} and \cite{Hille-Ploog1}.

\begin{remark} \label{rem:symmetric}
The opposite algebra of a quiver algebra is given by reverting all arrows. Therefore, $\cA_t \cong \cA_t\op$. In particular, the abelian categories of left and right $\cA_t$-modules are equivalent.  
\end{remark}

\subsection{The projective-injective module $\bm{P(t)}$, and the rank of $\bm{\cA_t}$-modules}

By a straightforward calculation of the injective hull and of the projective cover for the simple module $S(t)$, one finds $P(t)=I(t)$. Moreover, $P(t)$ is the unique indecomposable projective-injective $\cA_t$-module. We have $\udim(P(t)) = (1,2,\ldots,t)$.

\begin{definition} The \emph{rank} of an $\cA_t$-module $M$ is
 \[ \rk M \coloneqq  \dim M_t = \dim \Hom(P(t),M) = \dim \Hom(M, I(t)) \,. \]
\end{definition}

Note that by definition, the rank is additive on short exact sequences of $\cA_t$-modules, and thus induces $\rk\colon K_0(\cA_t) \to \IZ$.

\begin{remark}
The subcategory of $\lmod{\cA_t}$ of rank 0 modules is equivalent to the module category over the preprojective algebra of $\kk A_{t-1}$, and in particular abelian.
\end{remark}

\subsection{Euler pairing and quadratic form}

We recall the projective resolutions of the simple $\cA_t$-modules. 
In particular, $S(1),\ldots,S(t-1)$ have projective dimension two, whereas $S(t)$ has projective dimension one. Hence $\cA_t$ has global dimension 2.

\begin{lemma} \label{lem:resolutions}
The simple representations $S(1),\ldots,S(t)$ have projective resolutions
\[ \begin{array}{rl}
  P(1) \too P(2) \too P(1) \too S(1) \\
  P(i) \too P(i-1) \oplus P(i+1) \too P(i) \too S(i) & \text{for $i=2, \ldots, t-1$}, \\
  P(t-1) \too P(t) \too S(t)
\end{array} \]
\end{lemma}

The \emph{Euler pairing} of two $\cA_t$-modules $M$ and $N$ is defined as
 \[ \chi(M,N) \coloneqq \hom(M,N) - \ext^1(M,N) + \ext^2(M,N) \, . \]
$\chi(M,N)$ only depends on the classes of $M$ and $N$ in the Grothendieck
group $K_0(\cA_t) \cong \IZ^t$.
The associated \emph{quadratic form} is given by $q(M) \coloneqq \chi(M,M)$.

\begin{lemma} \label{lem:euler-pairing}
Writing $\udim(M) = (m_1,\ldots,m_t)$ and $\udim(N) = (n_1,\ldots,n_t)$,
\[ \chi(M,N) = m_t n_t + \sum_{i=1}^{t-1} m_i(n_i-n_{i+1}) + n_i(m_i-m_{i+1}) \, . \]
\end{lemma}

\begin{proof}
Given any representation $M$, we can apply the resolutions of \hyref{Lemma}{lem:resolutions} via the horseshoe lemma to a composition series of $M$, and obtain a projective resolution for $M$:
\[
\bigoplus_{i=1}^{t-1} P(i) \otimes M_i \too
\bigoplus_{i=2}^{t-1} P(i-1) \otimes M_i \oplus P(i+1) \otimes M_i \too
\bigoplus_{i=1}^{t}   P(i) \otimes M_i \too M
\]
Using the functor $\Hom(\blank,N)$ on this resolution produces a complex with terms
\[
\bigoplus_{i=1}^t\Hom(M_i,N_i) \too \bigoplus_{i=2}^t \Hom(M_i,N_{i-1}) \oplus \Hom(M_{i-1},N_i) \too \bigoplus_{i=1}^{t-1} \Hom(M_i,N_i)
\]
whose cohomologies are $\Ext^i(M,N)$.
Hence
\[ \chi(M,N) = \sum_{i=1}^t m_i n_i - \sum_{i=2}^t (m_i n_{i-1} + m_{i-1} n_i) + \sum_{i=1}^{t-1} m_i n_i \]
which translates to the formula of the proposition.
\end{proof}

\begin{corollary} \label{cor:quadratic-form}
The quadratic form of $\cA_t$ is positive definite:
\[ q(d_1,\ldots,d_t) = d_1^2 + (d_1 - d_2)^2 + (d_2 - d_3)^2 + \ldots + (d_{t-1} - d_t)^2 \, . \] 
\end{corollary}

\subsection{Extensions and projective resolutions}

We next show that homomorphisms and extensions among $\cA_t$-modules are computed explicitly by the complex 
\[ \cC(M,N) \colon \cC^0(M,N) \xrightarrow{d^0} \cC^1(M,N) \xrightarrow{d^1} \cC^2(M,N) \]
for $M,N\in\lmod{\cA_t}$, with the terms already seen in the proof of \hyref{Lemma}{lem:euler-pairing}
\[
\bigoplus_{i=1}^t\Hom(M_i,N_i) \xrightarrow{d^0} \bigoplus_{i=2}^t \Hom(M_i,N_{i-1}) \oplus \Hom(M_{i-1},N_i)
\xrightarrow{d^1} \bigoplus_{i=1}^{t-1} \Hom(M_i,N_i)
\]
where the differentials are induced by the arrows of the representions in $M$ and $N$:
\[ \begin{array}{r@{~=~}l@{\hspace{2em}\text{for~}}l}
 d^0(f_i)   & (\beta^N_{i-1}f_i + f_i\beta^M_i, -\alpha^N_i f_i - f_i\alpha^M_{i-1}) & f_i\colon M_i\to N_i,\\[0.7ex]
 d^1(g_i,0) & \alpha^N_{i-1}g_i + g_i\alpha^M_{i-1}                                  & g_i \colon M_i\to N_{i-1},\\[0.7ex]
 d^1(0,h_i) & \beta^N_i h_i + h_i\beta^M_i                                           & h_i\colon M_{i-1}\to N_i.
\end{array} \]

Thus, $\cC(M,N)$ becomes a complex of vector spaces. Note that $\cC(M,N)$ is not symmetric: $\cC^2(M,N) = \cC^0(M,N) \oplus \Hom(M_t, N_t)$. The vector space in the vertex $t$ plays an exceptional role.

\begin{lemma}
The functors
  $H^i(\cC(\blank,N)) \colon \lmod{\cA_t}\op \to \lmod{\cA_t}$
form a $\delta$-functor, for a fixed $\cA_t$-module $N$.
\end{lemma}

\begin{proof}
By construction, $\cC(M,N)$ is covariantly functorial in $N$ and contravariantly functorial in $M$. An exact sequence $M'\to M\to M''$ of modules induces an exact sequence $\cC(M'',N) \to \cC(M,N) \to \cC(M',N)$ of complexes. Hence a short exact sequence $0 \to M' \to M \to M'' \to 0$ yields a long exact sequence of modules

\medskip
\noindent
\resizebox{\textwidth}{!}{%
$0 \to \cH^0(M'') \to \cH^0(M) \to \cH^0(M') \xxto{\delta}
       \cH^1(M'') \to \cH^1(M) \to \cH^1(M') \xxto{\delta}
       \cH^2(M'') \to \cH^2(M) \to \cH^2(M') \to 0$}

\medskip
\noindent
where we set $\cH^i(\blank) \coloneqq H^i(\cC(\blank,N))$. This long exact cohomology sequence means that the $H^i(\cC(\blank,N))$ form a contravariant $\delta$-functor.
\end{proof}

\begin{proposition} \label{prop:ext-spaces}
$\Ext^i(M,N) = H^i(\cC(M,N))$ for $M,N \in \lmod{\cA_t}$ and $i\in\IZ$.
\end{proposition}

\begin{proof}
Comparing the differential $d^0\colon \cC^0(M,N)\to\cC^1(M,N)$ with morphisms of $\cA_t$-representations gives $H^0(\cC(M,N)) = \ker(d^0) = \Hom(M,N)$ right away. Next we claim that $H^i(\cC(\blank,N))$ is an erasable $\delta$-functor, i.e.\ $H^i(\cC(P,N))=0$ for projective $\cA_t$-modules and $i>0$. For this, first take $P=P(j)$ and $N=S(l)$. Then
\[ \cC(P(j),S(l)) = [P(j)_l^* \xxto{d^0} P(j)_{l+1}^* \oplus P(j)_{l-1}^* \xxto{d^1} P(j)_l^*] \]
(with the obvious modifications if $l=1$ or $l=t$). To check that $d^1$ is surjective, we use three facts: generally for path algebras, the projective representation $P(j)$ corresponding to a vertex $j$ of the quiver are obtained by following arrows emanating out of $j$; the special form of the quiver for $\cA_t$; and the definition of $d^1$.

Since the simple modules generate $\lmod{\cA_t}$ via extensions, the surjectivity of $d^1$ for all $N$ follows, i.e.\ $H^2(\cC(P,N))=0$.
We then get $H^1(\cC(P,N))=0$ because the Euler pairings for $\Hom^\bullet(M,N)$ and for $\cC(M,N)$ coincide. Finally, the claim follows from the universal property of derived functors \cite[\S2]{Weibel}.
\end{proof}

\subsection{Inequalities for $\bm{{\Hom}}$ and $\bm{{\Ext}}$ dimensions}

\begin{proposition} \label{prop:inequalities}
Let $M,N$ be $\cA_t$-modules. Then
\begin{enumerate}
\item $\hom(M,N) \geq \ext^2(N,M)$.
\item $\hom(M,N) = \ext^2(N,M)$ if and only if $\rk M=0$ or $\rk N=0$.
\item If $\rk M=\rk N=1$ and $\hom(M,N) \geq 2$, then $\ext^2(N,M) \geq 1$.
\end{enumerate}
\end{proposition}

For a geometric proof of the crucial inequality (1), see \cite{Hille-Ploog3}.

\begin{proof}
(1) Consider the maximal symmetric subcomplex $\cS(M,N) \subset \cC(M,N)$, with terms
  $\cS^0(M,N) \coloneqq \bigoplus_{i=1}^{t-1} \Hom(M_i,N_i) \subseteq \cC^0(M,N)$ and
  $\cS^p(M,N) \coloneqq \cC^p(M,N)$ for $p=1,2$,
and the differentials of $\cC(M,N)$.
Then $\cS(M,N)$ is selfdual in the sense $\cS(M,N) \cong \cS(N,M)^*[-2]$; for the isomorphism one needs to change one sign in the differential. Therefore, the following inclusion yields the claim:
\[ \Ext^2(N,M)^* = H^2(\cS(N,M))^* \cong H^0(\cS(M,N)) \subseteq H^0(\cC(M,N)) = \Hom(M,N) \, . \]

(2) follows from (1) and the construction of $\cS(M,N)\subseteq\cC(M,N)$.

(3) is similar: $\cC^0(M,N) \cong \cC^2(N,M)^*\oplus\Hom(M_t,N_t)$. With $\rk M = \rk N = 1$, this discrepancy amounts to 1 between $\Hom$ and $\Ext^2$.
\end{proof}

\subsection{Serre functor and Calabi--Yau objects}

Because $\cA_t$ has finite global dimension, its bounded derived category $\Db(\cA_t)$ has a Serre functor $\SSS$ or, equivalently, Auslander--Reiten sequences (see \cite[Thm.~I.2.4]{Reiten-vandenBergh} for the equivalence). Therefore, $\SSS$ is the derived Nakayama functor and $\SSS = \tau[1]$ where $\tau$ is the AR translation. See theorem and proof on page 37 of \cite{Happel}.

An object $M \in \Db(\cA_t)$ is called \emph{$d$-Calabi--Yau} (often abbreviated to $d$-CY) if $\SSS(A) \cong A[d]$, where $d\in\IZ$.

\begin{proposition} \label{prop:CY-objects}
The projective-injective module $P(t) \in \Db(\cA_t)$ is $0$-CY, and
the simple modules $S(1),\ldots,S(t-1) \in \Db(\cA_t)$ are $2$-CY.  
\end{proposition}

\begin{proof}
There is the unique indecomposable projective-injective module $P(t)=I(t)$, which is thus fixed by the Nakayama functor: $\SSS P(t) = P(t)$.

The projective resolutions of the simple modules $S(1),\ldots,S(t-1)$ are given by \hyref{Lemma}{lem:resolutions}. Their injective resolutions are ---as is easily checked by hand--- $S(1) \to I(1) \to I(2) \to I(1)$ and $S(i) \to I(i) \to I(i-1) \oplus I(i+1) \to I(i)$ for $i=2,\ldots,t-1$. Hence the Nakayama functor $\SSS$ sends $S(i)$ to $S(i)[2]$, so that these simples are $2$-CY.
\end{proof}

\begin{remark}
It is actually true that the triangulated category $\Db_0(\cA_t)$ of rank 0 objects is a $2$-CY category, i.e.\ it has Serre functor $[2]$. For a geometric proof of this fact, see \cite{Hille-Ploog3}. It does not follow formally from the proposition. 
\end{remark}

\subsection{The modules $\Delta(i)$ and $\nabla(i)$}

For any $i\in\{1,\ldots,t\}$, we define the $\cA_t$-representations $\nabla(i)$ as follows: $\nabla(i)_j=0$ for $j<i$ and $\nabla(i)_j=\kk$ for $j\geq i$, with
  $\xymatrix@1@C=4em{\nabla(i)_j & \ar[l]|{\beta_j=1} \nabla(i)_{j+1}}$
for $i\leq j<t$, and all other maps zero.

The representation $\Delta(i)$ has the same vector spaces as $\nabla(t+1-i)$ on all vertices, but with $\alpha_j=1$ wherever possible.
See \hyref{Example}{ex:worms} for visualisations of these modules. As a mnemonic: $\dim \Delta(i) = \dim \nabla(i) = i$.

These modules form chains of injections $\Delta(1) \into \Delta(2) \into \cdots \into \Delta(t)$ and of surjections $\nabla(t) \onto \nabla(t-1) \onto \cdots \onto \nabla(1)$. We will see in \hyref{Section}{sec:sequences} that these are full exceptional sequences,
  $\Delta \coloneqq (\Delta(1),\ldots,\Delta(t))$ and
  $\nabla \coloneqq (\nabla(t),\ldots,\nabla(1))$.

\section{Exceptional and spherical modules} \label{sec:exceptionals}

\noindent
In this section we classify all exceptional and all $2$-spherical modules
over $\cA_t$. Recall that an object $M\in\Db(\cA_t)$ is called \emph{exceptional} if $\Hom^\bullet(M,M) = \kk$. And $M$ is called \emph{$e$-spherical} if $\SSS(M) \cong M[e]$ and $\Hom^\bullet(M,M) = \kk \oplus \kk[-e]$. Also recall $\rk M = \dim M_t$.

\begin{definition} A module $M$ is called \emph{thin} if $\udim(M)$ has entries only zero or one.
\end{definition}

\begin{theorem} Let $M$ be an $\cA_t$-module. \label{thm:module-classification}
  \begin{enumerate}
  \item $M$ is exceptional if and only if $M$ is indecomposable, thin and $\rk(M) = 1$.
  \item $M$ is $2$-spherical if and only if $M$ is indecomposable, thin and $\rk(M) = 0$.
  \end{enumerate}
\end{theorem}

\begin{remark}
If $t=2$, then the module $P(2)$ is 0-spherical of rank 2.
\end{remark}

\begin{proof}
(1) Let $M$ be an exceptional module, with dimension vector $\underline d \coloneqq \dim(M)$. Its quadratic form is $q(\underline d) = 1$, and by \hyref{Corollary}{cor:quadratic-form} this happens precisely if $d_1 = d_2 = \ldots = d_a = 0$ and $d_{a+1} = \ldots = d_t = 1$. Hence $M$ is thin of rank 1.

Conversely, let $M$ be an indecomposable, thin module of rank 1. Computing endomorphisms of $M$ as a representation yields $\End(M)=\kk$. Now \hyref{Proposition}{prop:inequalities} implies $\ext^2(M,M) < \hom(M,M)=1$, by the assumption of $\rk(M)=1$. We finally get $\ext^1(M,M)=0$ from $\chi(M,M)=q(M)=1$, as $\udim(M)$ is of type $(0,\ldots,0,1\ldots,1)$; possibly without leading zeroes.

(2) Let now $M$ be a $2$-spherical module. We see $\rk(M)=0$ from
\[ \Hom(P(t),M) = \Hom(M[-2],P(t))^* = \Hom(P(t),M[-2]) = \Ext^{-2}(P(t),M) = 0 \]
where we have used Serre duality twice; first for the $2$-CY object $M$ and then for the 0-CY object $P(t)$; see \hyref{Proposition}{prop:CY-objects}. Moreover, we have $q(M)=2$ which, together with $\rk(M)=0$, forces
  $\dim(M) = (0,\ldots,0,1,\ldots,1,0,\ldots,0)$;
possibly without leading zeroes. In particular, $M$ is thin of rank 0.

Conversely, if $M$ is an indecomposable, thin module of rank 0, then we again find $\End(M)=\kk$ from checking represention endomorphisms. \hyref{Proposition}{prop:inequalities} now gives $\ext^2(M,M) = \hom(M,M) = 1$, and then $q(M)=2$ gets us $\Ext^1(M,M)=0$.
It remains to show that $M$ is $2$-CY. By \hyref{Proposition}{prop:CY-objects}, the simple rank 0 modules $S(1),\ldots,S(t-1)$ are $2$-CY. Any simple, rank 0 module $M$ is a consecutive extension of $S(1),\ldots,S(t-1)$. Therefore, we can assume that $M$ occurs in an extension $0 \to M' \to M \to S(i) \to 0$ with $\SSS(M') \cong M'[2]$. Moreover, because $M$ is thin, $\ext^1(S(i),M')=1$. Hence, applying the autoequivalence $\SSS$ to the essentially unique (up to scalars) extension leads to $\SSS(M)\cong M[2]$.
\end{proof}

\begin{corollary} \label{cor:count}
  An $\cA_t$-module is a brick if and only if it is indecomposable and thin.
  Moreover, over $\cA_t$ there are $2^t-1$ exceptional modules, $2^{t+1}-t-2$ bricks, and the number
  of $2$-spherical modules is $2^t-t-1$.
\end{corollary}

\begin{proof}
  $(\Longleftarrow)$ By the theorem, indecomposable thin modules are either 2-spherial (rank 0) or exceptional (rank 1), and therefore bricks.

  $(\Longrightarrow)$ We have
$2 = 2\hom(M,M) \geq \hom(M,M)+\ext^2(M,M) = q(M) \geq 1$
for a brick $M$, using \hyref{Proposition}{prop:inequalities} (1). If $q(M)=1$, then $M$ is thin by \hyref{Corollary}{cor:quadratic-form}. Otherwise, $q(M)=2$ and $\hom(M,M)=\ext^2(M,M)$, hence $\rk(M)=0$ by \hyref{Proposition}{prop:inequalities} (2), and again $M$ has to be thin.

  The dimension vector of an exceptional module is $(0,\ldots,0,1,\ldots,1)$, possibly with no leading zeroes. For each map in the representation, there are two choices, $\alpha$ or $\beta$. Thus there are $2^{i-1}$ exceptional modules of dimension $i$. Altoghether, there are $\sum_{i=1}^t 2^{i-1} = 2^t-1$ exceptional modules.

  Any brick, i.e.\ an indecomposable thin module, has the dimension vector of some exceptional $\cA_l$-module, for an $l\in\{1,\ldots,t\}$. Hence their number is given by the sum $\sum_{l=1}^t(2^l-1) = 2^{t+1}-t-2$.
  Finally, the number of $2$-spherical modules is the difference $(2^{t+1}-t-2)-(2^t-1) = 2^t -t-1$.
\end{proof}

The spherical twist functors $\smash{\TTT_{S(i)}}$ associated to the 2-spherical simple modules $S(1),\ldots,S(t-1)$ generate a subgroup $\braid{t} \coloneqq \langle \TTT_{S(1)},\ldots,\TTT_{S(t-1)}\rangle \subset \Aut(\Db(\cA_t))$. This subgroup is isomorphic to the braid group on $t$ strands; see \hyref{Appendix}{app:spherical}.

\begin{corollary}
  Let $M$ be a $2$-spherical $\cA_t$-module. Then there exists $\beta\in\braid{t}$ with $\beta(M) = S(1)$.
\end{corollary}

\begin{proof}
  By the theorem, $M$ is a thin module. Let $i\geq1$ be minimal with $M_i\neq0$, and $j<t$ maximal with $M_j\neq0$. If $i\neq j$, consider two cases:

  (i) $M_i = \kk \xxto{\alpha} \kk = M_{i+1}$: Then $\Hom(M,S(i))=\kk$, hence $\Ext^2(M,S(i)) = \Hom(S(i),M)^* = 0$, and finally $\Ext^1(M,S(i))=0$ from $\chi(M,S(i))=1$. Put $M' \coloneqq \smash{\TTT_{S(i)}^{\,-1}}(M)$. Plugging $\Hom^\bullet(M,S(i))=\kk$ into the triangle describing the inverse spherical twist functor (see \hyref{Appendix}{app:spherical}), we obtain a short exact sequence $0 \to M' \to M \to S(i) \to 0$.

  (ii) $M_i = \kk \xxfrom{\beta} \kk = M_{i+1}$: Then $\Hom(S(i),M)=\kk$, and $\Ext^1(S(i),M)=\Ext^2(S(i),M))=0$ as above. Put $M' \coloneqq \TTT_{S(i)}(M)$. The spherical twist triangle reduces to the short exact sequence $0 \to S(i) \to M \to M' \to 0$.

  In both cases, $M'$ is a 2-spherical module supported on vertices $i+1,\ldots,j$. Repeating this process, we see that a combination of spherical twists along $S(i)$ together with their inverses send $M$ to a simple module. In the final step, we move this simple via $\braid{t}$ to $S(1)$, using $\TTT_{S(i)}(S(i)) = S(i)[-1]$ and
  \[ \TTT_{S(i)}(S(i+1)) = \smextension{S(i)}{S(i+1)} \eqqcolon S^{(i)} \qquad\text{and}\qquad
     \TTT_{S^{(i)}}(S(i+1)) = S(i)[-1] \, . \qedhere \]
\end{proof}

\begin{lemma} \label{lem:exceptional-pairs}
If $(M,N)$ is an exceptional pair of $\cA_t$-modules, then $\Ext^2(M,N) = 0$ and $\hom(M,N) \leq 1$.  
\end{lemma}

\begin{proof}
The first claim follows immediately from the inequality \hyref{Proposition}{prop:inequalities} (1): $\ext^2(M,N) \leq \hom(N,M) = 0$.

For the second statement, $\hom(M,N)\geq2$ would lead to $\ext^2(N,M)\geq1$ by \hyref{Proposition}{prop:inequalities} (3), i.e.\ an extension in the wrong direction. Here we have used that exceptional modules have rank 1 by \hyref{Theorem}{thm:module-classification}.
\end{proof}

\subsection{Representing thin modules by worms} \label{sub:worms}
A non-zero, indecomposable thin representation of $\cA_t$ is a sequence of maps
 $\kk \smash{\xxto{\alpha}} \kk$ and $\kk \smash{\xxfrom{\beta}} \kk$.
Therefore, it is uniquely encoded by a word in the letters $\alpha$ and $\beta$, together with the last index of a non-zero vector space in the representation.
For exceptional modules, the encoding is particularly simple, since the index of the last non-zero vector space is always $t$.

We will depict these modules using the following convention: we read the word in the letters $\alpha$ and $\beta$ from left to right, and $\alpha$ is drawn as a line going right and $\beta$ is drawn as a line going up. This we call a \emph{worm}.

\begin{example} \label{ex:worms}
The seven exceptional $\cA_3$-modules, as representations and worms:
\[ \begin{array}{ *{2}{r @{~[} c@{}c@{}c@{}c@{}c @{]} @{\hspace{1em}} l} }
   S(3) = \nabla(1) = \Delta(1) = &   0 &      &   0 &      & \kk & \worm{\draw (0,0) node {};}
\\[1.5ex]
   \Delta(2) = &   0 &      & \kk & \xra & \kk & \worm{\draw (0,0) node {} -- (1,0) node {};} 
&  \hspace{2.5em}
   \nabla(3) = & \kk & \xlb & \kk & \xlb & \kk & \smash{\worm{\draw (0,0) node {} -- (0,1) node {} -- (0,2) node {};}}
\\[1ex]
  \Delta(3) = & \kk & \xra & \kk & \xra & \kk & \worm{\draw (0,0) node {} -- (1,0) node {} -- (2,0) node {};}
&  \nabla(2) = &   0 &      & \kk & \xlb & \kk & \worm{\draw (0,0) node {} -- (0,1) node {};}
\\[1ex]
               & \kk & \xra & \kk & \xlb & \kk & \worm{\draw (0,0) node {} -- (1,0) node {} -- (1,1) node {};} 
&              & \kk & \xlb & \kk & \xra & \kk & \worm{\draw (0,0) node {} -- (0,1) node {} -- (1,1) node {};}
\end{array} \]
\end{example}

\section{Full exceptional sequences of modules} \label{sec:sequences}

\subsection{Worm diagrams} \label{sub:worm-diagrams}
We now define worm diagrams as certain collections of worms. Worms will always be conflated with exceptional modules, as explained in \hyref{Subsection}{sub:worms}. 
We consider $\IZ\times\IZ$ as a lattice grid in the obvious way. 

\begin{definition}
A \emph{worm diagram of size $t$} is a graph with the following properties:
  \begin{enumerate}
  \item the vertices exhaust the triangle $\{(m,n) \in \IZ\times\IZ \mid m+n \leq t \text{ and } m,n\geq 1 \}$,
  \item the edges lie on the lattice grid,
  \item the connected components are $t$ worms of lengths $1,2,\ldots,t$, respectively.
  \end{enumerate}
\end{definition}

\begin{proposition} \label{prop:worm-diagrams}
Every worm diagram of size $t$ gives rise to a full exceptional sequence of $\cA_t$-modules.
\end{proposition}

\begin{proof}
For an exceptional module $E$, denote by $\mu_\alpha(E), \mu_\beta(E) \in \{0,\ldots,t-1\}$ the numbers of consecutive $\alpha$-maps ending in the vertex $t$, and of consecutive $\beta$-maps going out of vertex $t$, respectively.

Let now $(E^1,\ldots,E^t)$ be a worm diagram, where the $i$-th worm $E^i$ is the one containing the point $(i,i)$. We consider $E^i$ as an exceptional $\cA_t$-module. If $i<j$, then by the shape of worm diagrams, we have $\mu_\beta(E^i) \geq \mu_\beta(E^j)$. For any $\lambda\in\kk$, the map
  $\smash{(E^i)_t=\kk \xxto{\lambda} \kk=(E^j)_t}$
of the vector spaces of $E^i$ and $E^j$ in the vertex $t$ can be uniquely extended to the left: by $\lambda$ for each vertex where the maps in $E^i$ and $E^j$ coincide, and by $0$ once they differ.
\[ \xymatrix@R=4ex{
  \cdots                & \kk \ar[l]_-\beta \ar@{..>}[d]^-0  & \kk \ar[l]_-\beta \ar@{..>}[d]^-\lambda & \kk \ar[l]_-\beta \ar@{..>}[d]^-\lambda & \kk \ar[l]_-\beta \ar[d]^-\lambda \\
  \cdots \ar[r]_-\alpha & \kk \ar[r]_-\alpha & \kk               & \kk \ar[l]^-\beta                & \kk \ar[l]^-\beta \\
} \]
Likewise, we have $\mu_\alpha(E^i) \leq \mu_\alpha(E^j)$, and we have unique leftward extension of
  $\smash{(E^i)_t=\kk \xxto{\lambda} \kk=(E^j)_t}$
from the final vertex, in the same way.

This argument shows $\Hom(E^i,E^j)=\kk$ and, in the same way, $\Hom(E^j,E^i)=0$.

We now get $\Ext^2(E^i,E^j) = \Ext^2(E^j,E^i) = 0$ from \hyref{Proposition}{prop:inequalities}: the first vanishing uses the inequality $\ext^2(M,N) \leq \hom(N,M)$. Equality occurs precisely if $\rk M=0$ or $\rk N=0$, giving the second vanishing, since $\rk E^i=\rk E^j=1$.

Finally, we have $\chi(E^i,E^j)=\chi(E^j,E^i)=0$, using the Euler pairing for $\cA_t$, (\hyref{Lemma}{lem:euler-pairing}), and the dimension vectors of $E^i$ and $E^j$ (\hyref{Theorem}{thm:module-classification}). Hence
 $\Ext^1(E^i,E^j) \cong \Hom(E^i,E^j) = \kk$ and $\Ext^1(E^j,E^i) = 0$,
and this proves that $(E^1,\ldots,E^t)$ is an exceptional sequence.

The sequence is full because the dimension vectors $\udim(E^1),\ldots,\udim(E^t)$ span $\IZ^t = K_0(\cA_t)$.
\end{proof}

\begin{example}
We show the six worm diagrams of size 3:

\smallskip
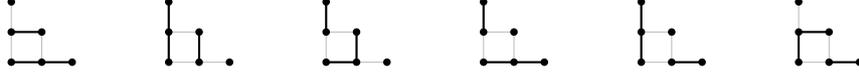
\begin{center}
\wormgrid{\draw (0,2) node {}; \draw (0,1) node {} -- (1,1) node {};
          \draw (0,0) node {} -- (1,0) node {} -- (2,0) node {};}
\hspace{2em}
\wormgrid{\draw (2,0) node {}; \draw (1,0) node {} -- (1,1) node {};
          \draw (0,0) node {} -- (0,1) node {} -- (0,2) node {};}
\hspace{2em}
\wormgrid{\draw (2,0) node {}; \draw (0,2) node {} -- (0,1) node {};
          \draw (0,0) node {} -- (1,0) node {} -- (1,1) node {};}
\hspace{2em}
\wormgrid{\draw (1,1) node {}; \draw (0,2) node {} -- (0,1) node {};
          \draw (0,0) node {} -- (1,0) node {} -- (2,0) node {};}
\hspace{2em}
\wormgrid{\draw (1,1) node {}; \draw (2,0) node {} -- (1,0) node {};
          \draw (0,0) node {} -- (0,1) node {} -- (0,2) node {};}
\hspace{2em}
\wormgrid{\draw (0,2) node {}; \draw (2,0) node {} -- (1,0) node {};
          \draw (0,0) node {} -- (0,1) node {} -- (1,1) node {};}
\end{center}
By the proposition, each worm diagram corresponds to a full exceptional sequence of $\cA_3$-modules.
We will prove in \hyref{Proposition}{prop:exceptional-sequences} that in fact all full exceptional sequences
of modules come from worm diagrams.

There are $6=3!$ worm diagrams of size 3. By \hyref{Remark}{rem:action}, the symmetric group on $t$ letters acts simply-transitively on exceptional sequences of $\cA_t$-modules.
\end{example}

\subsection{Filtrations of the projective-injective module}

\begin{proposition} \label{prop:exceptional-sequences}
If $F^1 \subset \ldots \subset F^t = P(t)$ is a filtration such that
\begin{enumerate}
\item each $F^i$ is indecomposable of rank $i$ and
\item each graded piece $E^i \coloneqq F^i/F^{i-1}$ is thin,
\end{enumerate}
then there is a permutation $\sigma\in\symm{t}$ such that $(E^{\sigma(1)}, E^{\sigma(2)}, \ldots, E^{\sigma(t)})$ is a full exceptional sequence of $\cA_t$-modules.

Moreover, any full exceptional sequence of $\cA_t$-modules occurs in this way.
\end{proposition}

\begin{proof}
$(\Longrightarrow)$ Assume that $F^\cdot$ is a filtration as in the theorem. Then each graded piece $E^i \coloneqq F^i/F^{i-1}$ is thin and indecomposable of rank 1, hence an exceptional module by \hyref{Theorem}{thm:module-classification}. Since $\udim(P(t))=(1,2,\ldots,t)$, the filtration induces a worm diagram of size $t$, with worms $E^{\sigma(1)},\ldots,E^{\sigma(t)}$, where $E^{\sigma(i)}$ is the worm ending in $(i,t-i)$. By \hyref{Proposition}{prop:worm-diagrams}, $(E^{\sigma(1)},\ldots,E^{\sigma(t)})$ is a full exceptional sequence.

$(\Longleftarrow)$ Given an exceptional sequence of $\cA_t$-modules $\cE = (E^1,\ldots,E^t)$, form its iterated universal extension $T = T^1\oplus\cdots\oplus T^t$ via
  $0 \to T^{i-1} \to T^i \to (E^{t+1-i})^{\oplus c_i} \to 0$
for some $c_i\geq1$. Then $T$ is a tilting module because $\Ext^2(E^i,E^j) = 0$ for all $i,j$ by \hyref{Lemma}{lem:exceptional-pairs}.
  From the construction of the iterated universal extensions (as opposed to coextensions), we have $T^1 \subset \cdots \subset T^t$.
The maximal summand $T^t$ is the universal extension of $(E^1)^{\oplus a_1},\ldots,(E^t)^{\oplus a_t}$, where the numbers $a_i\geq1$ are dimensions of certain $\Ext^1$-spaces.

Because $P(t)$ is the unique indecomposable injective-projective $\cA_t$-module, it must occur as a summand of any tilting object. This forces $a_1=\cdots=a_t=1$ (and subsequently $c_i=1$ at all steps), i.e.\ $T^t=P(t)$.
Hence we find the smaller direct summands as Jordan--H\"older subquotients for the filtration $F^i=T^i$ of $P(t)$.
Note $\udim(P(t)) = (1,2,\ldots,t)$, corresponding to the triangle grid of worm diagrams.
\end{proof}

\begin{corollary} \label{cor:worm-diagrams}
There is a bijection between worm diagrams of size $t$ and full exceptional sequences of $\cA_t$-modules.  
\end{corollary}

\begin{corollary} \label{cor:hom}
  Let $(E^1,\ldots,E^t)$ be a full exceptional sequence of $\cA_t$-modules. Then
  $\hom(E^i,E^j) = \ext^1(E^i,E^j) = 1$ and $\ext^2(E^i,E^j) = 0$ for any $i<j$.
\end{corollary}

\begin{proof}[Proof of \hyref{Corollary}{cor:hom}]
  The Hom and Ext dimensions among worms have been computed in the proof of \hyref{Proposition}{prop:worm-diagrams}. These dimensions then apply to all full exceptional sequences of modules, by the previous corollary.
\end{proof}

\begin{example}
Consider the standard exceptional sequence $\Delta = (\Delta(1),\Delta(2),\Delta(3))$ of $\lmod{\cA_3}$.
Its associated iterated universal extension and the worm diagram are

\medskip
\parbox{0.7\textwidth}{
$ T = T^1 \oplus T^2 \oplus T^3
     = \Delta(3) \oplus \extension{\Delta(2) \\ \Delta(3)} \oplus
       \extension{\Delta(1) \\ \Delta(2) \\ \Delta(3)}
$
}
\hfill
\parbox{0.3\textwidth}{
  \wormgrid{\draw (0,2) node {}; \draw (0,1) node {} -- (1,1) node {};
            \draw (0,0) node {} -- (1,0) node {} -- (2,0) node {};}
}
\smallskip

\noindent
We mention in passing that the maps $\Delta(1) \to \Delta(2) \to \Delta(3)$ are all injective, so that condition $\ASS$ of \cite[\S1.3]{Hille-Ploog1} is met, which means that $T$ is an exact tilting object.
\end{example}

\smallskip
\begin{example}
For $t=2$, we have $T = T^1 \oplus T^2$ with
  $T^2 = P(2) = \smash{[\xymatrix@1{\kk\: \ar@<0.5ex>[r]^{(1~0)^t} & \:\kk^2 \ar@<0.5ex>[l]^{(0~1)}}]}$. \newline
There are two filtrations meeting the conditions of \hyref{Proposition}{prop:exceptional-sequences}:
\[ \begin{array}{@{}llllll@{}}
F^1 = E^1 \coloneqq \Delta(2) = [\xymatrix@C=1.5em@1{\kk\: \ar@<0.7ex>[r]|1 & \:\kk \ar@<0.7ex>[l]|0 }] \subset T^2, &
E^2 = S(2), & \cE = (E^2,E^1), & \sigma = (12) &
\wormgridtwo{\draw (0,1) node {}; \draw (1,0) node {} -- (0,0) node {};} \\[1ex]

F^1 = E^1 \coloneqq \nabla(2) = [\xymatrix@C=1.5em@1{\kk\: \ar@<0.7ex>[r]|0 & \:\kk \ar@<0.7ex>[l]|1 }] \subset T^2, &
E^2 = S(2), & \cE = (E^1,E^2), & \sigma = \id &
\wormgridtwo{\draw (1,0) node {}; \draw (0,1) node {} -- (0,0) node {};}
\end{array} \]
\end{example}

\section{Group actions on exceptional sequences}

\noindent
Denote by $\Exc{t}$ the set of full exceptional sequences in $\Db(\cA_t)$, up to isomorphism, and by $\mExc{t}$ the subset of sequences of modules. Let $\braid{t}$ be the braid group on $t$ strands, and $\symm{t}$ the symmetric group of $t$ letters. There is a canonical surjective homomorphism $\braid{t}\to\symm{t}$.

The braid group $\braid{t}$ acts in two ways on $\Exc{t}$: first, exceptional sequences can be mutated (we will always deal with right mutations in this article). Second, the $2$-spherical modules $S(1),\ldots,S(t-1)$ form an $A_{t-1}$-chain, and thus give rise to a  $\braid{t}$-action on the whole derived category (see \hyref{Subsection}{sub:braids}), and in particular on $\Exc{t}$. Moreover, we will see that the symmetric group $\symm{t}$ acts simply-transitively on $\mExc{t}$ from the left and from the right, in a combinatorial fashion.

The $\symm{t}$-action on $\mExc{t}$ does not extend to $\Exc{t}$, and the $\braid{t}$-actions on $\Exc{t}$ do not restrict to $\mExc{t}$. Nevertheless, we will prove that the two braid group actions lift the symmetric group actions in a natural way. In order to see this, we introduce a count of all vertical edges in a worm diagram:
\begin{align*}
  f\colon \mExc{t} &\to \{0,1,\ldots,{\textstyle\binom{t}{2}}\}, \\
   \cE = (E^1,\ldots,E^t) &\mapsto f(\cE) \text{ is the number of $\beta$-maps among $E^1,\ldots,E^t$.} 
\end{align*}
Clearly, minimum and maximum are uniquely achieved by $f(\Delta)=0$ (no vertical edges) and $f(\nabla)=\binom{t}{2}$ (all edges vertical), respectively.

\subsection{Symmetric group action} \label{sub:symmetric-group-action}

We define two permutations $\sigma(\cE),\lambda(\cE)\in\symm{t}$ for any full exceptional sequence $\cE$ of $\cA_t$-modules. Mostly, we employ $\sigma(\cE)$.

\begin{definition}
Let $\cE = (E^1,\ldots,E^t)$ be a full exceptional sequence of $\cA_t$-modules.
\begin{enumerate}
\item The \emph{(start) permutation} of $\cE$ is $\sigma(\cE) \coloneqq (\sigma(E^1),\ldots,\sigma(E^t)) \in \symm{t}$, where $\sigma(E^i) \coloneqq t+1-\dim(E^i)$ is the starting vertex of $E^i$.
\item The \emph{length permutation} of $\cE$ is $\lambda(\cE) \coloneqq (\dim E^1,\ldots,\dim E^t) \in \symm{t}$.
\end{enumerate}
\end{definition}

\begin{example}
For the $\Delta$-modules, $\sigma(\Delta(i))=t+1-i$. Thus their start permutation $\sigma(\Delta) = (t,\ldots,2,1) = \omega \in\symm{t}$ is the longest word.
\end{example}

\begin{lemma} \label{lem:permutations}
Let $\cE = (E^1,\ldots,E^t)$ be a full exceptional sequence of $\cA_t$-modules.

\begin{enumerate}
\item $\lambda(\cE) = \omega\cdot\sigma(\cE)$, where $\omega \coloneqq (t,\ldots,2,1) \in \symm{t}$
\item For $i<j$ fixed: $\sigma(i) < \sigma(j) \iff (E^i)_{j-1} \xxfrom{\beta} (E^i)_j$ is a vertical edge.
\item $f(\cE) = \#\{ (i,j) \mid 1\leq i < j \leq t, \sigma(i) < \sigma(j) \}$
\end{enumerate}
\end{lemma}

\begin{proof} (1) At once from $\lambda(i) = \dim E^i$, $\sigma(i) = t+1-\dim E^i$ and $\omega(i) = t+1-i$.

(2) Assume $i<j$ and $\sigma(i)<\sigma(j)$. Consider the worm subdiagram of size $j$. It has boundary diagonal $(0,j),(1,j-1),\ldots,(j,0)$ which intersects precisely $j$ worms. Note that only the head (starting vertex) of $E^j$ occurs in the subdiagram; this is the point $(j,j)$. This forces all edges left of $j$ along the boundary to be vertical.

(3) follows from (2).  
\end{proof}

\begin{theorem} \label{thm:exceptional-bijections}
For fixed $t\in\IN$, there are bijections between the following sets:
\begin{enumerate}
\item Full exceptional sequences of $\cA_t$-modules.
\item Ascending filtrations $F^\cdot$ of $P(t)$ with $\rk F^i=i$ and $F^i/F^{i-1}$ thin.
\item Worm diagrams of size $t$.
\item The symmetric group $\symm{t}$.
\end{enumerate}
\end{theorem}

\begin{proof}
The bijection $(1) \longleftrightarrow (2)$ was established in \hyref{Proposition}{prop:exceptional-sequences} and the proof of that statement contained the bijection $(2) \longleftrightarrow (3)$, as mentioned in \hyref{Corollary}{cor:worm-diagrams}.

Mapping a full exceptional sequence of modules or, equivalently, the corresponding worm diagram to its permutation is obviously injective. Moreover, every permutation comes from a worm diagram: given $\sigma\in\symm{t}$, we start drawing a worm diagram with the left-most worm, which begins in $(0,t)$ and goes $t+1-\sigma(1)$ steps downwards. Given a partially completed worm diagram, the $i$-th worm starts in $(i,t-i)$ and is of length $t+1-\sigma(i)$; its shape is determined by the worms already drawn and the definition of worm diagrams.
\end{proof}

\begin{remark} \label{rem:action}
Assigning start permutations induces left and right actions of the symmetric group on full exceptional sequences of modules. By the above theorem, these actions are simply-transitive.
\end{remark}

\begin{example} \label{ex:t=3-worm-diagrams}
We consider again the six worm diagrams of size 3. Below each worm diagram, we give its start permutation, and we show how to move between the diagrams using simple left translations, e.g.\ $(12)\cdot\binom{123}{321}=\binom{123}{312}$.

\smallskip
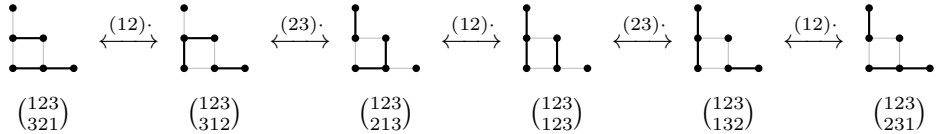
\begin{center}

\wormgridc{\draw (0,2) node {}; \draw (0,1) node {} -- (1,1) node {};
           \draw (0,0) node {} -- (1,0) node {} -- (2,0) node {};}
          {\wormcap{321}}
\ltranspose{12}
\wormgridc{\draw (0,2) node {}; \draw (2,0) node {} -- (1,0) node {};
           \draw (0,0) node {} -- (0,1) node {} -- (1,1) node {};}
          {\wormcap{312}}
\ltranspose{23}
\wormgridc{\draw (2,0) node {}; \draw (0,2) node {} -- (0,1) node {};
           \draw (0,0) node {} -- (1,0) node {} -- (1,1) node {};}
          {\wormcap{213}}
\ltranspose{12}
\wormgridc{\draw (2,0) node {}; \draw (1,0) node {} -- (1,1) node {};
           \draw (0,0) node {} -- (0,1) node {} -- (0,2) node {};}
          {\wormcap{123}}
\ltranspose{23}
\wormgridc{\draw (1,1) node {}; \draw (2,0) node {} -- (1,0) node {};
           \draw (0,0) node {} -- (0,1) node {} -- (0,2) node {};}
          {\wormcap{132}}
\ltranspose{12}
\wormgridc{\draw (1,1) node {}; \draw (0,2) node {} -- (0,1) node {};
           \draw (0,0) node {} -- (1,0) node {} -- (2,0) node {};}
          {\wormcap{231}}
\end{center}
\end{example}

\subsection{Braid group action from right mutations} \label{sub:braids-mutation}

Our next aim is to categorify the right action, using the well-known braid group action on $\Exc{t}$ from mutating exceptional sequences; see \cite{Rudakov}.

Given an exceptional pair $(E',E)$ in $\Db(\cA_t)$, its
\emph{right mutation} $(E,\RR E')$ is again an exceptional pair, defined by the canonical triangle
\begin{align*}
   \RR E' \too E' \too \Hom^\bullet(E',E)^*\otimes E
\end{align*}
where we slightly deviate from the standard definition, by taking
$\RR E$ as the cocone of the canonical morphism rather than the cone.

An exceptional sequence $\cE = (E^1,\ldots,E^t)$ in $\Db(\cA_t)$ has a \emph{right mutation at $E^i$},
\begin{align*}
  \RR_i\cE &\coloneqq (E^1,\ldots,E^{i-1}, E^{i+1}, \RR E^i,E^{i+2},\ldots,E^t) \, ,
\end{align*}
and it is well-known that $\RR_i\cE$ is again an exceptional sequence, so that we get $\RR_i\colon\Exc{t}\isom\Exc{t}$; the inverses are left mutations. Mutations satisfy the braid relations, 
  $\RR_i\RR_{i+1}\RR_i \cE \cong \RR_{i+1}\RR_i\RR_{i+1} \cE$,
leading to an action
$\braid{t} \to \Exc{t}$.

\begin{lemma} \label{lem:right-mutation}
Let $(E',E)$ be an exceptional pair of $\cA_t$-modules with $\hom(E',E)=\ext^1(E',E)=1$ and such that non-zero morphisms $E' \to E$ are surjective. Then the right mutation $\RR E'$ is the module given by the extension
\[ 0 \too E \too \RR E' \too \ker(E'\to E) \too 0 \, . \]
\end{lemma}

\begin{proof}
Using $\Ext^2(E',E)=0$ from \hyref{Lemma}{lem:exceptional-pairs}, the triangle defining $\RR E'$ yields the following exact sequence of $\cA_t$-modules:
\begin{align*}
%   0 \to H^{-1}(\LL E) \to \Hom(E',E) \otimes E' \to E \to
%         H^0(\LL E) \to \Ext^1(E',E) \otimes E' \to 0 \, , \\
  0 \to \Ext^1(E',E)^* \otimes E \to H^0(\RR E') \to E' \xxto{\varepsilon}
        \Hom(E',E)^* \otimes E \to H^1(\RR E') \to 0 \, . 
\end{align*}
The assumption $\Hom(E',E) = \Ext^1(E',E) = \kk$ simplifies the exact sequence to
  $0 \to E \to H^0(\RR E') \to E' \xxto{\varepsilon} E \to H^1(\RR E') \to 0$.
The other assumption implies that the evaluation morphism $\varepsilon$ is surjective, so that $H^1(\RR E')=0$. Hence $\RR E'$ is indeed a module, sitting in the stated extension.
\end{proof}

\hyref{Lemma}{lem:right-mutation} applies to full exceptional sequences $\cE=(E^1,\ldots,E^t)$ of modules, because $\hom(E^i,E^j)=\ext^1(E^i,E^j)=1$ for all $i\leq j$ by \hyref{Corollary}{cor:hom}. For adjacent modules, we have the

\begin{lemma} \label{lem:injective-surjective}
  Let $\cE = (E^1,\ldots,E^t)$ be a full exceptional sequence of $\cA_t$-modules.
  Then non-zero morphisms $E^i\to E^{i+1}$ are either injective or surjective.
\end{lemma}

\begin{proof}
This follows from the equivalence of worm diagrams and full exceptional sequences of modules, \hyref{Corollary}{cor:worm-diagrams}: walking along adjacent worms $E^i,E^{i+1}$ from the diagonal towards the origin of the worm diagram, they run in parallel until one of them stops. If $E^i$ stops first, i.e.\ $\dim E^i<\dim E^{i+1}$, then it embeds into $E^{i+1}$. Whereas if $E^{i+1}$ stops first, i.e.\ $\dim E^i>\dim E^{i+1}$, then $E^i$ surjects onto it.
\end{proof}

\begin{corollary} \label{cor:right-mutations}
Let $\cE = (E^1,\ldots,E^t) \in \mExc{t}$ and $i\in\{1,\ldots,t-1\}$. Then the following conditions are equivalent:
\begin{enumerate}
\item $\RR_i\cE \in \mExc{t}$.
\item there exists a surjection $E^i \onto E^{i+1}$.
\item the worm $E^i$ is longer than the worm $E^{i+1}$.
\end{enumerate}
\end{corollary}

 Clause (3) of the following statement says that right mutations $\RR_i$ categorify right multiplications by $\tau_i$.

\begin{proposition} \label{prop:mutating-sequences}
  Let $\cE = (E^1,\ldots,E^t)$ be a full exceptional sequence of $\cA_t$-modules with $f(\cE)\geq1$.
  If $E^i$ is longer than $E^{i+1}$, then
  \begin{enumerate}
    \item $\RR_i\cE$ is a full exceptional sequence of $\cA_t$-modules,
    \item $f(\RR_i\cE) = f(\cE)-1$,
    \item $\sigma(\RR_i\cE) = \sigma(\cE)\cdot\tau_i$.
  \end{enumerate}
\end{proposition}

\begin{proof}[Proof of \hyref{Proposition}{prop:mutating-sequences}]
By $f(\cE)\geq1$, there exists an $i$ such that $E^i\onto E^{i+1}$. Then (1) is the content of \hyref{Corollary}{cor:right-mutations}.

On (2): $\RR_i$ modifies only the module $E^i$, so we compare this module to $\RR E^i$. Now $E^i$ surjecting onto $E^{i+1}$ means that the worm $E^{i+1}$ sits as a copy in the worm $E^i$, i.e.\ the representations have the same $\alpha/\beta$ maps in degrees $j,\ldots,t$. Let $j \coloneqq \sigma(E^{i+1}) < \sigma(E^j)$ be the starting vertex of the shorter worm. Then $(E^i)_{j-1} \xxfrom{\beta} (E^i)_j$ is necessarily a vertical edge. In the extension defining $\RR E^i$, this map is changed to $\alpha$, i.e.\ $(\RR E^i)_{j-1} \xxto{\alpha} (\RR E^i)_j$.
Hence $f(\RR_i\cE) = f(\cE) - 1$.

On (3): $\sigma(\RR_i\cE)$ has the same values as $\sigma(\cE)$, with $\sigma(\cE)(i)$ and $\sigma(\cE)(i+1)$ interchanged. This amounts to pre-composition $\sigma(\RR_i\cE) = \sigma(\cE)\cdot\tau_i$.
\end{proof}

We are ready to show that right mutations can transform any exceptional sequence of modules into the standard sequence $\Delta$. 

\begin{theorem} \label{thm:mutating-to-deltas}
Let $\cE$ be a full exceptional sequence of $\cA_t$-modules with $f(\cE)\geq1$ and write its length permutation $\lambda(\cE) = \tau_{i(f)}\cdots\tau_{i(1)}$ as a minimal product of simple transpositions.
Setting recursively $\cE_0 \coloneqq \cE$ and $\cE_j \coloneqq \RR_{i(j)}\cE_{j-1}$, each $\cE_j$ is a full exceptional sequence of $\cA_t$-modules, and $\cE_f = \Delta$.   
\end{theorem}

\begin{proof}
The basic step to mutate $\cE$ towards $\Delta$ is this equivalence: for $1\leq i<t$,
\[ f(\sigma(\cE)\cdot\tau_i) = f(\cE)-1 \iff E^i\onto E^{i+1} \]
The two worm diagrams $\sigma(\cE) = (E^1,\ldots,E^t)$ and $\sigma(\cE)\cdot\tau_i \eqqcolon (F^1,\ldots,F^t)$ are identical, except for the pairs $(E^i,E^{i+1})$ and $(F^i,F^{i+1})$. Each pair occupies the same space in the worm diagram grid. The shorter worms have the same form, whereas the longer worms differ. There is one more vertical edge if the longer worms comes first. Since $\dim E^i>\dim E^{i+1}$ translates to $E^i\onto E^{i+1}$ by \hyref{Lemma}{lem:injective-surjective}, the equivalence is established.

By \hyref{Proposition}{prop:mutating-sequences}, any surjection $E^i\onto E^{i+1}$ gives rise to the mutated exceptional sequence $\RR_i\cE$ of modules with one less vertical edge. In this way we can proceed to obtain a sequence of mutations $\RR_{i(f)}\cdots\RR_{i(1)}\cE = \Delta$ ending in the standard sequence with $f(\Delta)=0$. The number of mutations is $f = f(\cE)$. Then
\begin{align*}
 \omega &= \sigma(\Delta) = \sigma(\RR_{i(f)}\cdots\RR_{i(1)}\cE)
         = \sigma(\cE)\cdot\tau_{i(1)}\cdots\tau_{i(f)} ,
\end{align*}
using $\sigma(\RR_i\cE) = \sigma(\cE)\cdot\tau_i$ from \hyref{Proposition}{prop:mutating-sequences}.
Hence $\tau_{i(f)}\cdots\tau_{i(1)} = \omega\sigma(\cE) = \lambda(\cE)$.

These two arguments combine to the statement of the theorem.
\end{proof}

\begin{example} \label{ex:t=3-mutations}
  We illustrate right mutations with our running example $t=3$.
  Below each worm diagram, we show its start permutation.
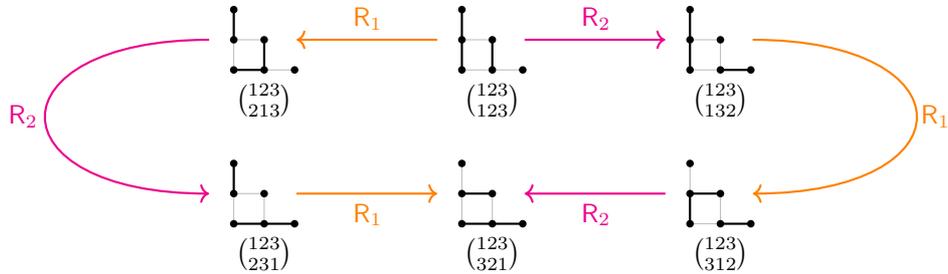
\begin{center}
\begin{tikzpicture}
  \newcommand*{\hh}{1.7} % parameter for distance between the two rows
% f=0:
    \thorm{ 0}{\hh}{123}{{0,0,0,1,0,2, 1,0,1,1, 2,0}}
% f=1:
    \thorm{-3}{\hh}{213}{{0,0,1,0,1,1, 0,1,0,2, 2,0}}
    \thorm{ 3}{\hh}{132}{{0,0,0,1,0,2, 1,0,2,0, 1,1}}
% f=2:
    \thorm{-3}{0}{231}{{0,0,1,0,2,0, 0,1,0,2, 1,1}}
    \thorm{ 3}{0}{312}{{0,0,0,1,1,1, 1,0,2,0, 0,2}}
% f=3:
    \thorm{ 0}{0}{321}{{0,0,1,0,2,0, 0,1,1,1, 0,2}}
% arrows
    \tikzset{R1/.style={->,thick,orange}}    
    \tikzset{R2/.style={->,thick,magenta}}
    \tikzset{shortb <>/.style={ shorten >=#1, shorten <=#1 } }     % shorten start and end by same amount
    \tikzset{shorts <>/.style={ shorten >=#1*0.5, shorten <=#1 } } % shorten start by half only
    \tikzset{shorte <>/.style={ shorten >=#1, shorten <=#1*0.5 } } % shorten end by half only
    \newcommand*{\tikzl}{0.6cm}
    \newcommand*{\tikzh}{\tikzl*0.5}
    \draw[R1,shorts <>=\tikzl] (123) -- (213) node[pos=0.55,above] {$\RR_1$};
    \draw[R2,shorte <>=\tikzl] (123) -- (132) node[pos=0.45,above] {$\RR_2$};
    \draw[R2,shortb <>=\tikzl] (213) .. controls +(left:3cm+\tikzl) and +(left:3cm+\tikzl) .. (231) node[midway,left] {$\RR_2~$};
    \draw[R1,shortb <>=\tikzh] (132) .. controls +(right:3cm+\tikzh) and +(right:3cm+\tikzh) .. (312) node[midway,right] {$\RR_1$};
    \draw[R1,shorte <>=\tikzl] (231) -- (321) node[pos=0.45,below] {$\RR_1$};
    \draw[R2,shorts <>=\tikzl] (312) -- (321) node[pos=0.55,below] {$\RR_2$};
\end{tikzpicture}
\end{center}
\end{example}

\subsection{Braid group action from spherical twists} \label{sub:braids-twisting}
Now we turn to the left action of $\symm{t}$ on $\mExc{t}$.
We use spherical twists functors (see \hyref{Appendix}{app:spherical} for details) along the simple modules of rank 0, for which we introduce shorthand notation:
\[ \TTT_i \coloneqq \TTT_{S(i)} \colon \Db(\cA_t) \isom \Db(\cA_t) \qquad\text{for } i=1,\ldots,t-1.\]
Since the $\TTT_i$ are autoequivalences, out of any exceptional sequence $\cE=(E^1,\ldots,E^t)$ in $\Db(\cA_t)$, we get another one: $\TTT_i\cE \coloneqq (\TTT_i(E^1),\ldots,\TTT_i(E^t))$. In particular, each twist yields a bijection $\TTT_i\colon \Exc{t} \isom \Exc{t}$.
Spherical twists satisfy the braid relations, 
  $\TTT_i\TTT_{i+1}\TTT_i \cE \cong \TTT_{i+1}\TTT_i\TTT_{i+1} \cE$,
leading to an action $\braid{t} \to \Exc{t}$.

In the following counterpart to \hyref{Theorem}{thm:mutating-to-deltas}, again a permutation is decomposed into simple transpositions, but instead of $\lambda(\cE) = \omega\sigma(\cE)$ it is now $\omega\sigma(\cE)^{-1}$.

\begin{theorem} \label{thm:twisting-worm-diagrams}
Let $\cE=(E^1,\ldots,E^t)$ be a full exceptional sequence of $\cA_t$-modules,
and write the permutation $\omega\sigma(\cE)^{-1} = \tau_{i(f)} \cdots \tau_{i(2)} \tau_{i(1)}$ as a minimal product of simple transpositions. Setting recursively $\cE_0 \coloneqq \cE$ and $\cE_j \coloneqq \TTT_{i(j)}(\cE_{j-1})$, each $\cE_j$ is a full exceptional sequence of $\cA_t$-modules, and $\cE_f = \Delta$.
\end{theorem}

Note that while right mutations of worm diagrams change the shape of exactly worm, a spherical twist may modify any number. Because of this, the combinatorial details become slightly more involved.

Fix an exceptional module $E$ and recall that $E$ is thin, i.e.\ has dimension vector $(0,\ldots,0,1,\ldots,1)$, possibly without any zeroes at the front. We think of $E$ as a worm crawling from the bottom left towards the top right, ending in the vertex $t$.

Recall that $E$ \emph{starts at $i$} if $\sigma(E)=i$, i.e.\ $E_i\neq0$ and $E_{i-1}=0$.
Also note that $\dim(E)$ is the length of the worm, hence 
  $E_i \neq 0 \iff \dim(E) \geq t+1-i \iff \sigma(E)\geq i$.
Let us introduce some graphical terminology:

\parbox{0.9\textwidth}{
\begin{itemize}
\item $E$ has a \emph{horizontal start at $i$} if $E_{i-1}=0$ and $E_i \xxto{\alpha} E_{i+1}$.
      \hfill \tikz[scale=0.3,thick]{\draw (0,0)--(1,0); \draw[fill] (0,0) circle (3pt); }
\item $E$ has a \emph{vertical start at $i$} if $E_{i-1}=0$ and $E_i \xxfrom{\beta} E_{ia+1}$.
      \hfill \tikz[scale=0.3,thick]{\draw (0,0)--(0,1); \draw[fill] (0,0) circle (3pt); }      
\item $E$ has a \emph{left hook at $i$} if $E_{i-1} \xxfrom{\beta} E_i \xxto{\alpha} E_{i+1}$.
      \hfill \tikz[scale=0.3,thick]{\draw (0,0)--(0,1)--(1,1); \draw[fill] (0,1) circle (3pt); }
\item $E$ has a \emph{right hook at $i$} if $E_{i-1} \xxto{\alpha} E_i \xxfrom{\beta} E_{i+1}$. 
      \hfill \tikz[scale=0.3,thick]{\draw (0,0)--(1,0)--(1,1); \draw[fill] (1,0) circle (3pt); }
\end{itemize}
}

\begin{lemma} \label{lem:simple-to-exceptional}
Let $E$ be an exceptional module, and let $S(i)$ be a simple module of rank 0, i.e.\ $i\in\{1,\ldots,t-1\}$.
Then $\dim\Ext^p(S(i),E)\leq1$ for all $p$, and
\begingroup
\addtolength{\jot}{-0.9ex}
\begin{align*}
\quad \Hom(S(i),E)   \neq 0 &\iff \text{$E$ has a vertical start or a right hook at $i$;} \\
      \Ext^2(S(i),E) \neq 0 &\iff \text{$E$ has a horizontal start or a left hook at $i$.} 
\intertext{If $\Ext^2(S(i),E) = 0$, then}
      \Ext^1(S(i),E) \neq 0 &\iff  \text{either $E$ starts at $i+1$ or $E$ has a right hook at $i$.}
\end{align*}
\endgroup
\end{lemma}

\begin{proof}
Considering $S(i)$ and $E$ as representations computes $\Hom(S(i),E)$ readily, and likewise by Serre duality $\Ext^2(S(i),E) = \Hom(E,S(i))^*$, using that $S(i)$ is $2$-spherical according to \hyref{Proposition}{prop:CY-objects}:
\[ \begin{array}{l@{~}c@{~}l}
 \Hom(S(i),E)  &=& \begin{cases}
                     \kk, & \text{if } E_i \xxfrom{\beta} E_{i+1},
                            \text{ and } E_{i-1}=0 \text{ or } E_{i-1} \xxto{\alpha} E_i \\
                     0    & \text{else}    
                    \end{cases}
\\[4ex]
 \Ext^2(S(i),E) &=& \begin{cases}
                     \kk, & \text{if } E_i \xxto{\alpha} E_{i+1},
                            \text{ and } E_{i-1}=0 \text{ or } E_{i-1} \xxfrom{\beta} E_i \\
                     0    & \text{else}    
                    \end{cases}
\end{array} \]
The formulas translate to the first two statements made in the proposition.

Put $s \coloneqq t+1-\dim(E)$, so that $E_s\neq 0$ and $E_{s-1}=0$.
We claim that, under the assumption $\Ext^2(S(i),E)=0$,
\[
\ext^1(S(i),E) = \begin{cases}
                 0, & \text{if } i \leq s-2 \\
                 1, & \text{if } i =    s-1 \\
                 0, & \text{if } i =    s   \\
      \hom(S(i),E), & \text{if } i \geq s+1 \\
                 \end{cases}
\]
To see this, start with the Euler pairing
  $\ext^1(S(i),E) = \hom(S(i),E) - \chi(S(i),E)$,
making use of the assumption.
In the next step, we invoke \hyref{Lemma}{lem:euler-pairing} to write $\chi(S(i),E) = 2e_i-e_{i-1}-e_{i+1}$, where $\udim(E) = (e_1,\ldots,e_t)$. In the first two cases of the statement, $\hom(S(i),E)=0$ from the formula above.

If $i=s$, i.e.\ the worm starts at vertex $i$, then we have $\hom(S(i),E)=1$ because the assumption $\ext^2(S(i),E)=0$ means that the first non-zero map of $E$ is $\beta$. Then $\ext^1(S(i),E)=1-(2-0-1)=0$.
\end{proof}

\begin{corollary} \label{cor:twisting-sequences}
Let $\cE = (E^1,\ldots,E^t) \in \mExc{t}$ and $i\in\{1,\ldots,t-1\}$. Then the following conditions are equivalent:
\begin{enumerate}
\item $\TTT_i\cE \in \mExc{t}$.
\item $\Ext^2(S(i),E^j) = 0$ for all $j=1,\ldots,t$.
\item The worm starting at $i$ does so vertically.
\end{enumerate}
\end{corollary}

\begin{proof}
  $(1) \iff (2)$ is \hyref{Corollary}{cor:twisting-exceptional-modules}, and $(2) \iff (3)$ is \hyref{Lemma}{lem:simple-to-exceptional}, since a worm starting vertically at $i$ prevents other worms from having left hooks at $i$.
%
% Consider a worm diagram and the diagonal $(0,i),(1,i-1),\ldots,(i,0)$ inside it. Necessarily, this diagonal crosses $i$ worms. Hence each diagonal contains precisely one worm head. In particular, a diagonal with a vertical start cannot contain a horizontal start. 
% Moreover, the diagonal through $(0,i)$ cannot contain both a vertical start and the midpoint of a left hook. Otherwise the adjacent diagonal, through $(0,i+1)$ cannot cut through $i+1$ worms.
% These two properties imply that a vertical start of some worm is neither a horizontal start nor a left hook of any other worm.
\end{proof}

Recall that $f(\cE)$ was defined as the number of vertical edges among all worms.
We now show that a suitable spherical twist of an exceptional sequence of modules reduces the $f$-invariant by one. Moreover, the twist $\TTT_i$ categorifies the left multiplication by the simple transpositions $\tau_i \coloneqq (i,i+1) \in \symm{t}$.

\begin{proposition} \label{prop:twisting-worms}
If $\cE = (E^1,\ldots,E^t)$ is a full exceptional sequence of $\cA_t$-modules with $f(\cE)\geq1$ and the worm starting at $i$ does so vertically, then
\begin{enumerate}
\item $\TTT_i\cE$ is a full exceptional sequence of $\cA_t$-modules,
\item $f(\TTT_i\cE) = f(\cE) - 1$,
\item $\sigma( \TTT_i\cE ) = \tau_i \cdot \sigma(\cE)$.
\end{enumerate}
\end{proposition}

\begin{proof}
(1) is the content of \hyref{Corollary}{cor:twisting-sequences}.

We turn to the computation of $f(\TTT_i\cE)$. Let $E=E^j$ for one $j$.
According to \hyref{Lemma}{lem:simple-to-exceptional}, there are the following four possibilities for $\Hom(S(i),E)$ and $\Ext^1(S(i),E)$, for which \hyref{Lemma}{lem:module-twist} gives the exact sequence containing $\TTT_i E$:
\[ \begin{array}{cccr@{\:}l}
 \text{case} & \Hom(S(i),E) & \Ext^1(S(i),E) & \multicolumn{2}{c}{\text{exact sequence for } \TTT_i E} \\[1ex]
 \text{(O)}  & 0            & 0            &          0 \to & E \to \TTT_i E \to 0 \\
 \text{(H)}  & \kk          & 0            & 0 \to S(i) \to & E \to \TTT_i E \to 0 \\
 \text{(E)}  & 0            & \kk          &          0 \to & E \to \TTT_i E \to S(i) \to 0 \\
 \text{(HE)} & \kk          & \kk          & 0 \to S(i) \to & E \to \TTT_i E \to S(i) \to 0
\end{array} \]
We examine the cases separately. In (O), the exceptional module is unchanged.

In (H), $E$ starts vertically in $i$, and the spherical twist strips off the simple $S(i)$ from $E$. Thus, the number of vertical edges decreases by one.

In (E), the spherical twists prolongs $E$ by the simple $S(i)$ along $\kk \xxto{\alpha} E_i$ because of the morphism $\TTT_iE\to S(i)$. Here, the number of vertical edges is unchanged.

In (HE), the right hook at $i$ is replaced by a left hook at $i+1$. Again, the number of vertical arrows is unchanged.

Let $E^l$ be the worm starting at $i$. By assumption, $E^l$ has a vertical start. The shape of worm diagrams means that no worm has a left hook at $i$. Therefore, by \hyref{Lemma}{lem:simple-to-exceptional} case (H) occurs exactly once, and $f(\TTT_i\cE) = f(\cE)-1$.

Moreover, the proof of that proposition shows (E) occurs precisely if $E$ starts at vertex $i+1$; hence this case appears exactly once, too. We now relate $\sigma \coloneqq \sigma(\cE)$ and $\sigma'\ \coloneqq \sigma(\TTT_i\cE)$. Recall $\sigma(j) = \sigma(E^j) = t+1-\dim(E^j)$ is the starting vertex of the module $E^j$. By the above, there are exactly two positions $j,l$ where $\sigma$ and $\sigma'$ differ because cases (O) and (HE) do not change worm lengths, and cases (H) and (E) occur once each. Then $\sigma'(j) = i+1 = \sigma(j)+1$ for the unique $E^j$ starting in $i$; this is case (H). In turn for case (E), the module $E^l$ with starting vertex $i+1$ becomes prolonged by $S(i)$, i.e.\ $\sigma'(l) = i = \sigma(l)-1$.
Hence $\tau_i\cdot \sigma = (i,i+1)\cdot \sigma(\cE) = \sigma' = \sigma(\TTT_i\cE)$. So we have also proved (3).
\end{proof}

\begin{proof}[Proof of \hyref{Theorem}{thm:twisting-worm-diagrams}]
We first show the following equivalence, for any $1\leq i<t$:
\[ f(\tau_i\cdot\sigma(\cE)) = f(\cE)-1  \iff \text{the worm starting at $i$ does so vertically}. \]
Write $\cE = (E^1,\ldots,E^t)$ and let $E^k,E^l$ be the worms starting at $i$ and $i+1$, respectively, i.e.\ $\sigma(E^k)=i$ and $\sigma(E^l)=i+1$. By definition, $(F^1,\ldots,F^t) \coloneqq \tau_i\cdot\sigma(\cE)$ is the worm diagram obtained from $\cE$ by $\sigma(F^j)=i+1$ and $\sigma(F^l)=i$; all other worm lengths are unchanged. 
We examine what this means for worm shapes. 
First, if $\sigma(E^j)>i+1$, then $F^j=E^j$ is unchanged. 
Second, the worm $E^k$ starting at $i$ gets shorter by one edge, and the worm $E^l$ starting at $i+1$ gets longer by one edge. For all other worms, left hooks at $i+1$ are turned into right hooks, and vice versa. In particular, the number of vertical edges only depends on $E^k$ and $E^l$: if $E^k$ starts vertically, then $F^k$ loses that edge and $F^l$ gains a horizontally starting edge; hence $f(\tau_i\cdot\sigma(\cE)) = f(\cE) - 1$. On the other hand, if $E^k$ starts horizontally, then $F^l$ gains a vertically starting edge, so that $f(\tau_i\cdot\sigma(\cE)) = f(\cE) + 1$. This proves the claim.

By \hyref{Proposition}{prop:twisting-worms}, any vertically starting worm $E^i$ gives rise to the twisted exceptional sequence $\TTT_i\cE$ of modules with one less vertical edge. In this way we can proceed to obtain a sequence of $f=f(\cE)$ spherical twists $\TTT_{i(f)}\cdots\TTT_{i(1)}\cE = \Delta$ ending in the standard sequence. Then
\begin{align*}
 \omega &= \sigma(\Delta) = \sigma(\TTT_{i(f)}\cdots\TTT_{i(1)}\cE)
         = \tau_{i(f)}\cdots\tau_{i(1)}\cdot\sigma(\cE) ,
\end{align*}
using $\sigma(\TTT_i\cE) = \tau_i\cdot\sigma(\cE)$ from \hyref{Proposition}{prop:twisting-worms}.
Hence $\tau_{i(f)}\cdots\tau_{i(1)} = \omega\sigma(\cE)^{-1}$.

The theorem follows from this computation and the above equivalence.
\end{proof}

\begin{example} \label{ex:t=3-twists}
  We illustrate the proposition with our running example $t=3$.
  Note how this hexagon is different from the one of \hyref{Example}{ex:t=3-mutations}.
  Below each worm diagram, we show its start permutation. This example categorifies \hyref{Example}{ex:t=3-worm-diagrams}, replacing left multiplications with $\tau_i$ by spherical twists $\TTT_i$.
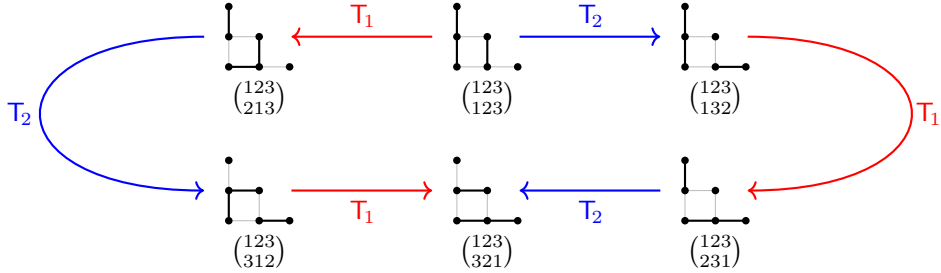
\begin{center}
\begin{tikzpicture}
  \newcommand*{\hh}{1.7} % parameter for distance between the two rows
% f=3:
    \thorm{ 0}{\hh}{123}{{0,0,0,1,0,2, 1,0,1,1, 2,0}}
% f=2:
    \thorm{-3}{\hh}{213}{{0,0,1,0,1,1, 0,1,0,2, 2,0}}
    \thorm{ 3}{\hh}{132}{{0,0,0,1,0,2, 1,0,2,0, 1,1}}
% f=1:
    \thorm{-3}{0}{312}{{0,0,0,1,1,1, 1,0,2,0, 0,2}}
    \thorm{ 3}{0}{231}{{0,0,1,0,2,0, 0,1,0,2, 1,1}}
% f=0:
    \thorm{ 0}{0}{321}{{0,0,1,0,2,0, 0,1,1,1, 0,2}}
% arrows
    \tikzset{T1/.style={->,thick,red}}    
    \tikzset{T2/.style={->,thick,blue}}
    \tikzset{shortb <>/.style={ shorten >=#1, shorten <=#1 } }     % shorten start and end by same amount
    \tikzset{shorts <>/.style={ shorten >=#1*0.5, shorten <=#1 } } % shorten start by half only
    \tikzset{shorte <>/.style={ shorten >=#1, shorten <=#1*0.5 } } % shorten end by half only
    \newcommand*{\tikzl}{0.6cm}
    \newcommand*{\tikzh}{\tikzl*0.5}
    \draw[T1,shorts <>=\tikzl] (123) -- (213) node[pos=0.55,above] {$\TTT_1$};
    \draw[T2,shorte <>=\tikzl] (123) -- (132) node[pos=0.45,above] {$\TTT_2$};
    \draw[T2,shortb <>=\tikzl] (213) .. controls +(left:3cm+\tikzl) and +(left:3cm+\tikzl) .. (312) node[midway,left] {$\TTT_2~$};
    \draw[T1,shortb <>=\tikzh] (132) .. controls +(right:3cm+\tikzh) and +(right:3cm+\tikzh) .. (231) node[midway,right] {$\TTT_1$};
    \draw[T1,shorte <>=\tikzl] (312) -- (321) node[pos=0.45,below] {$\TTT_1$};
    \draw[T2,shorts <>=\tikzl] (231) -- (321) node[pos=0.55,below] {$\TTT_2$};
\end{tikzpicture}
\end{center}
We observe the braid relation $\TTT_1\TTT_2\TTT_1(\nabla) \cong \TTT_2\TTT_1\TTT_2(\nabla)$. In fact, there is a functor isomorphism $\TTT_1\TTT_2\TTT_1 \cong \TTT_2\TTT_1\TTT_2$ on $\Db(\cA_t)$; see \hyref{Subsection}{sub:braids}.
\end{example}

\appendix

\section{Spherical twist functors} \label{app:spherical}

\noindent
Let $\Lambda$ be a finite-dimensional algebra of finite global dimension, and $\cD = \Db(\Lambda)$ the bounded derived category of left $\Lambda$-modules with Serre (Nakayama) functor $\SSS$.

An object $S$ of $\cD$ is \emph{$e$-spherical} if $\SSS(S)\cong S[e]$ and $\Hom^\bullet(S,S) = \kk \oplus \kk[-e]$. Consider the complex of bimodules $S^* \otimes_\kk S$; it corresponds to the endofunctor $\Hom^\bullet(S,\blank)\otimes S$ of $\cD$. There is the canonical evaluation morphism $\eta\colon S^* \otimes_\kk S \to \Lambda$, where $\Lambda$ is a bimodule in the natural way, corresponding to the identity functor. Denoting the functor associated to $\cone(\eta)$ by $\TTT_S$, we obtain a triangle of functors
\[ \Hom^\bullet(S,\blank) \otimes S \too \id \too \TTT_S \too \]
and $\TTT_S$ is called the \emph{spherical twist functor} along $S$.

By construction, we have $\TTT_S(S) \cong S[1-e]$, and $\TTT_S(M) \cong M$ for all $M\in S\orth = \{M \in \cD \mid \Hom^\bullet(S,M)=0 \}$. These two properties remind of reflections and can in fact be used to prove that $\TTT_S \colon \cD \isom \cD$ is an autoequivalence; see \cite[\S8]{Huybrechts}.
At one place, we will need the inverse functor, which is given by
\[ \TTT_S^{\,-1} \too \id \too \Hom^\bullet(\blank,S)^* \otimes S \too \, . \]

\subsection{Special case: modules}
We describe a situation when the spherical twist of a module is again a module.

\begin{lemma} \label{lem:module-twist}
Let $S$ be a simple module, and $M$ any module. Then 
 $H^1(\TTT_S(M)) = \Ext^2(S,M)\otimes S$.
Moreover, if $\Ext^{\geq2}(S,M)=0$, then $\TTT_S(M)$ is a module, and occurs in the exact sequence
\[ 0 \too \Hom(S,M)\otimes S \too M \too \TTT_S(M) \too \Ext^1(S,M)\otimes S \too 0 \, . \]
\end{lemma}

\begin{proof}
The triangle $\Hom^\bullet(S,M)\otimes S\to M \to \TTT_S(M) \to$ defines $\TTT_S(M) \in \Db(\Lambda)$. Its long exact cohomology sequence gives $H^1(\TTT_S(M)) = \Ext^2(S,M)\otimes S$ right away.

If $\Ext^{\geq2}(S,M)=0$, then $\Hom^\bullet(S,M) = \Hom(S,M) \oplus \Ext^1(S,M)[-1]$, and long exact cohomology sequence of the triangle is
\[ 0 \to H^{-1}(\TTT_S(M)) \to \Hom(S,M)\otimes S \xxto{\varphi} M \to H^0(\TTT_S(M))
     \to \Ext^1(S,M)\otimes S \to 0 \, . \]
Put $h \coloneqq \hom(S,M)$. The map $\varphi\colon S^{\oplus h} \to M$ is injective because it is the canonical evaluation and $S$ is simple. Thus $H^{-1}(\TTT_S(M))=0$ and $\TTT_S(M)$ is concentrated in degree 0, i.e.\ a module.
\end{proof}

\begin{corollary} \label{cor:twisting-exceptional-modules}
  Let $E$ be an exceptional $\cA_t$-module, and let $i\in\{1,\ldots,t-1\}$.
  If $\Ext^2(S(i),E)=0$, then $\TTT_{S(i)}(E)$ is an exceptional $\cA_t$-module.
\end{corollary}

\begin{proof}
The algebra $\cA_t$ has global dimension 2. Thus the only Ext vanishing of the lemma to be checked is in degree 2.
Hence $\TTT_{S(i)}(E)$ is a module by the lemma.

Finally note that images of exceptional objects under any fully faithful functor (e.g.\ an autoequivalence such as the spherical twist) are again exceptional. 
\end{proof}

\subsection{Braid relations} \label{sub:braids}
We briefly return to the general setting: if $S,S'\in\cD$ are $e$-spherical objects such that $\Hom^\bullet(S,S') = \kk[-n]$ for some $n$, then the twist functors $\TTT_S$ and $\TTT_{S'}$ satisfy the braid relations:
  $\TTT_S \TTT_{S'} \TTT_S \cong \TTT_{S'} \TTT_S \TTT_{S'}$. 
Clearly, this can be iterated to chains of $e$-spherical objects $S_1,\ldots,S_n$ such that $\dim\Hom^\bullet(S_i,S_j)=1$ if $|i-j|=1$ and zero else. Such a chain induces an action of the $n$-stranded braid group $\braid{n}$ on $\cD$, i.e.\ a group homomorphism
  $\braid{n} \to \Aut \cD$.

We apply this fact to the situation of this text: the $t-1$ simple modules $S(1),\ldots,S(t-1)$ are $2$-spherical objects of $\Db(\cA_t)$ and the only non-vanishing extensions among different simple modules are $\Ext^1(S(i),S(i+1))=\kk$. 

\begin{corollary} \label{cor:braid-relations}
  There is a braid group action $\braid{t} \to \Aut(\Db(\cA_t))$, mapping the braid intertwining strands $i$ and $i+1$ to $\TTT_{S(i)}$.
\end{corollary}

In fact, by \cite[Theorem~2.18]{Seidel-Thomas}, this action is actually effective.

\section{Dictionary algebra--geometry} \label{app:geometry}

\noindent
The algebra $\cA_t$ occurs in a geometric guise in our previous article \cite{Hille-Ploog1}. Let $X$ be a smooth, projective surface such that all line bundles are exceptional (this holds, for example, if $X$ is a rational, e.g.\ toric, surface), and let $C_1,\ldots,C_{t-1}$ be an $A_{t-1}$-chain of $(-2)$-curves in $X$, i.e.\ $C_i\cong\IP^1$ and $C_i^2=-2$ for all $i$. Then
\[ \cE = \big( \cO_X(-C_1-\cdots-C_{t-1}), \ldots, \cO_X(-C_2-C_1), \cO_X(-C_1), \cO_X \big) \]
is an exceptional sequence in $\Db(\Coh(X))$, and we also denote by $\cE$ the triangulated subcategory it generates. Define an additive category $\Coh_\cE(X) \coloneqq \cE \cap \Coh(X)$. Denote by $T$ the iterated universal extension of the exceptional sequence. Two facts are crucial for connecting these notions to this article:

First, $T\in\cE$ is a tilting bundle such that $\End(T) = \cA_t$, i.e.\ the algebra we study in this article occurs in a geometric tilting situation.
Second, the derived tilting equivalence $\RHom(T,\blank)\colon \cE \isom \Db(\cA_t)$ restricts to an equivalence of abelian categories $\Hom(T,\blank) \colon \Coh_\cE(X) \isom \lmod{\cA_t}$. This exactness property of the tilting equivalence actually holds more generally for self-intersection numbers $\leq -2$.

Note that the tilting functors $\RHom(T,\blank)$ and $\Hom(T,\blank)$ produce (complexes of) right modules whereas in this article we always consider left modules. However, this is not a concern because $\cA_t \cong \cA_t\op$ is symmetric; see \hyref{Remark}{rem:symmetric}.

We list some geometric counterparts to algebraic notions:

\bigskip\noindent
{\small
\begin{tabular}{@{} p{0.52\textwidth} @{\hspace{0.05\textwidth}} p{0.43\textwidth}}
  sheaf $F\in\Coh(X)\cap\cE$ & representation $M \in \lmod{\cA_t}$
\lfdict
  rank of the sheaf $F$ & $\rk M = \dim M_t = \hom(P(t),M)$
\lfdict
  locally free sheaf (vector bundle) $F$ &
  $M$ with all $M_{i-1}\to M_i$ injective
\lfdict
  torsion subsheaf of $F$ &
  maximal $M'\subseteq M$ with $\alpha^i M' = 0 \quad \forall i$
\lfdict
  exceptional sequence of line bundles \newline $(\cO_X(-C_1-\cdots-C_{t-1}), \ldots, \cO_X(-C_1), \cO_X)$
& exceptional sequence of $\Delta$-modules \newline $(\Delta(1),\Delta(2),\ldots,\Delta(t))$
\lfdict
  simple sheaves & simple modules
\\
  $\cO_{C_{t-1}}(-1),\ldots,\cO_{C_2}(-1),\cO_{C_1}$ torsion sheaves & $S(1),\ldots,S(t-2),S(t-1)$ 
\\
  $\cO_X(-C_1-\cdots-C_{t-1})$ & $S(t)$
\lfdict
  projective objects & projective modules
\\
  $\cO_X, \smextension{\cO_X(-C_1)}{\cO_X}, \ldots$ iterated extensions & $P(1),P(2),\ldots$
\end{tabular}
}

\end{document}